\documentclass[reqno]{amsart}

\usepackage[dvipsnames]{xcolor}

\usepackage{amsmath, amsfonts, amssymb, amsthm, amscd, cancel, enumerate, rotating, comment, appendix}
\usepackage{ulem}

\usepackage[capitalise, noabbrev]{cleveref}
\usepackage{paralist}

\usepackage{ulem}
\usepackage{times}

\usepackage[all]{xy}
\usepackage{enumerate}

\usepackage{tikz}
\usetikzlibrary{shapes}
\usetikzlibrary{snakes}

\usetikzlibrary{arrows}
\usepackage{tikz-cd}
\usetikzlibrary{matrix, decorations.pathmorphing}
\usetikzlibrary{decorations.markings}

\usetikzlibrary{calc}
\usepackage{graphicx}
\usepackage{pgf}
\usepackage{latexsym,amsmath,amssymb}
 
\newif\ifdviwin

\usepackage[latin1]{inputenc}
\usepackage[english]{babel}
\usepackage[mathscr]{eucal}
\usepackage{amsmath,amscd}
\usepackage{xxcolor}

\numberwithin{equation}{section}

\theoremstyle{plain}
\newtheorem{theorem}{Theorem}[section]

\newtheorem*{Main Theorem}{Main Theorem}

\newtheorem{proposition}[theorem]{Proposition}
\newtheorem{lemma}[theorem]{Lemma}
\newtheorem{corollary}[theorem]{Corollary}

{\Alph{theorem}}
{\Alph{theorem}}

\theoremstyle{definition}
\newtheorem{notation}[theorem]{Notation}
\newtheorem{Question}[theorem]{Question}

\newtheorem{construction}[theorem]{Construction}

\newtheorem{remark}[theorem]{Remark}
\newtheorem{definition}[theorem]{Definition}
\newtheorem{example}[theorem]{Example}

  \newcounter{numlist} %
  {\end{list}}%

\theoremstyle{remark}

\newtheorem{chunk}[theorem]{}
\newenvironment{bfchunk}{\begin{chunk}\textbf}{\end{chunk}}

\numberwithin{equation}{theorem}

\renewcommand{\emph}{\it}

\def\m{{\mathfrak m}}

\def\Tor{\operatorname{Tor}}
\def\im{\operatorname{im}}

\def\dim{\operatorname{dim}}

\def\dpth{\operatorname{depth}}
\def\htt{\operatorname{height}}

\def\im{\operatorname{im}}

\newcommand{\excise}[1]{}

\newcommand{\ba}{\mathbf{a}}
\newcommand{\bb}{\mathbf{b}}
\newcommand{\bc}{\mathbf{c}}
\newcommand{\bd}{\mathbf{d}}
\newcommand{\be}{\mathbf{e}}
\newcommand{\bff}{\mathbf{f}}

\newcommand{\bu}{{\bf u}}
\newcommand{\bx}{{\bf x}}

\newcommand{\bm}{{\mathbf{m}}}
\newcommand{\bma}{{\bm^{\ba}}}
\newcommand{\bmb}{{\bm^{\bb}}}

\newcommand{\bv}{{\mathbf{v}}}
\newcommand{\bva}{{\bv^{\ba}}}
\newcommand{\bvb}{{\bv^{\bb}}}
\newcommand{\bvc}{{\bv^{\bc}}}

\newcommand{\ol}{\overline}

\newcommand{\pd}{\operatorname{pd}}

\newcommand{\fm}{\mathfrak m}

\newcommand{\sfk}{\mathsf k}

\newcommand{\cC}{\mathcal C}
\newcommand{\ocC}{\ol{\cC}}

\newcommand{\cN}{\mathcal N}
\newcommand{\cNr}{\cN_r}

\newcommand{\supp}{\operatorname{Supp}}

\newcommand{\bw}{{\mathsf\Lambda}}

\newcommand{\lcm}{\operatorname{lcm}}

\newcommand{\ZZ}{\mathbb{Z}}
\newcommand{\RR}{\mathbb{R}}

\newcommand{\ssm}{\smallsetminus}
\newcommand{\qand}{\quad \mbox{and} \quad}
\newcommand{\qor}{\quad \mbox{or} \quad}
\newcommand{\qif}{\quad \mbox{if} \quad}
\newcommand{\qfor}{\quad \mbox{for} \quad}
\newcommand{\qwhere}{\quad \mbox{where} \quad}
\newcommand{\qwith}{\quad \mbox{with} \quad}
\newcommand{\qforsome}{\quad \mbox{for some} \quad}
\newcommand{\qforall}{\quad \mbox{for all} \quad}

\begin{document}
\bibliographystyle{amsplain}

\author[S.M.~Cooper]{Susan M. Cooper}
\address{Department of Mathematics\\
University of Manitoba\\
420 Machray Hall\\
186 Dysart Road\\
Winnipeg, MB\\
Canada R3T 2N2}
\email{susan.cooper@umanitoba.ca}

\author[S.~El Khoury]{Sabine El Khoury}
\address{Department of Mathematics,
American University of Beirut,
Bliss Hall 315, P.O. Box 11-0236,  Beirut 1107-2020,
Lebanon}
\email{se24@aub.edu.lb}

\author[S.~Faridi]{Sara Faridi}
\address{Department of Mathematics \& Statistics\\
Dalhousie University\\
6316 Coburg Rd.\\
PO BOX 15000\\
Halifax, NS\\
Canada B3H 4R2} 
\email{faridi@dal.ca}

\author[S~Mayes-Tang]{Sarah Mayes-Tang}
\address{Department of Mathematics\\
University of Toronto\\
40 St. George Street, Room 6290\\
Toronto, ON \\
Canada M5S 2E4}
\email{smt@math.toronto.edu}

\author[S.~Morey]{Susan Morey}
\address{Department of Mathematics\\
Texas State University\\
601 University Dr.\\
San Marcos, TX 78666\\U.S.A.}
\email{morey@txstate.edu}

\author[L.~M.~\c{S}ega]{Liana M.~\c{S}ega}
\address{Liana M.~\c{S}ega\\ Department of Mathematics and Statistics\\
   University of Missouri\\ \linebreak Kansas City\\ MO 64110\\ U.S.A.}
     \email{segal@umkc.edu}

\author[S.~Spiroff]{Sandra Spiroff }
\address{Department of Mathematics,
University of Mississippi,
Hume Hall 335, P.O. Box 1848, University, MS 38677
USA}

\email{spiroff@olemiss.edu}

\urladdr{http://math.olemiss.edu/sandra-spiroff/}

\title[Powers of Graphs]
{Powers of Graphs \& Applications to Resolutions of Powers of Monomial Ideals}

\begin{abstract}  
This paper is concerned with the question of whether geometric
  structures such as cell complexes can be used to simultaneously
  describe the minimal free resolutions of {\it all} powers of a
  monomial ideal. We provide a full answer in the case of square-free
  monomial ideals of projective dimension one, by introducing a
  combinatorial construction of a family of (cubical) cell complexes
  whose 1-skeletons are powers of a graph that supports the resolution
  of the ideal.

\end{abstract}
\maketitle 

\section{Introduction}

   The search for topological objects whose chain maps coincide
    with  free resolutions of a given monomial ideal has been a major
    research area in commutative algebra. Such topological objects are
    said to {\emph support} a (minimal) free resolution of the
    ideal. Diane Taylor~\cite{T} showed that every ideal in a polynomial ring
    generated by $q$ monomials has a free resolution supported on a
    $q$-simplex.

  Our work connects two fruitful directions of research in commutative
  algebra. On one hand, starting with Taylor's construction there has been a
  substantial body of work on finding smaller topological structures,
  such as simplical or more generally cell complexes, which support
  free resolutions of a given monomial ideal. We refer to
  Peeva~\cite[Chapter III]{P} for the basics of such constructions. On
  the other hand, the problem of studying the powers of a (not
  necessarily monomial) ideal $I$ arises naturally in the study of
  Rees algebras.  There are numerous analyses of invariants of the
  powers $I^r$ such as depth, regularity, projective dimension, Betti
  numbers, free resolutions, and more.  For a sampling, see Guardo and
  Van Tuyl \cite{GV}, Morey \cite{M}, Fouli and Morey \cite{FM},
  Engstr\"{o}m and Noren \cite{EN}.

At the intersection of the above mentioned directions of research lies the following: 

\begin{Question}\label{q:driving}
If $\Gamma$ is a cell complex supporting a minimal free resolution of
a monomial ideal $I$, can $\Gamma$ be used to define a family of cell
complexes $\{\Gamma^r\}_{r\ge 1}$ such that $\Gamma^r$ supports a
minimal free resolution of $I^r$ for each $r\ge 1$?
\end{Question}

In this paper we provide a positive answer to this question in the case when
$\Gamma=G$ is a tree, that is, a graph with no cycles, supporting a
minimal free resolution of a square-free monomial ideal $I$. Given a
positive integer $r$, we build a graph $G^r$, which can be viewed as a
power of $G$, using an abelianized extension of a graph-theoretic
construction called {\it box product} (\cref{s:graph product}). We
then use $G^r$ to construct a polyhedral cell complex $\ol{G^r}$,
which is shown to support a minimal free resolution of the ideal
$I^r$.

The polyhedral cell complex $\ol{G^r}$ is built on a skeleton that
originates from the graph $G^r$. More precisely, assuming that $G$ has
$q+1$ vertices, we describe an embedding of $G^r$ into $\mathbb R^{q}$
such that all the vertices of $G^r$ have non-negative integer
coordinates, and all edges have unit length and are parallel to one of
the standard basis vectors in $\mathbb R^q$.  The graph $G^r$ is no
longer a tree, except when $q \leq 1$ or $r=1$, as multiple cycles are
formed among its edges.  However, due to our embedding in $\mathbb
R^q$, the cycles are easily recognizable: they appear in $1$-skeletons
of cubes of various dimensions.  This is detailed in
Section~\cref{ss:embed}.  In Section~\cref {ss:graphpowers}, we
describe these cubes using an orientation on the edges that ensures
that the $1$-skeleton of each such cube has a source and a sink, which
can be used to identify the cube, and we define the cell complex
$\ol{G^r}$ as the collection of these cubes. \cref{p:poly-cell-cx}
proves that $\ol{G^r}$ is indeed a polyhedral cell complex.

On the other hand, it is known by Faridi and Hersey \cite{FH} that
every monomial ideal $I$ of projective dimension one has a minimal
free resolution supported on a graph $G$. When $I$ is square-free, we
label the vertices of the cell complex $\ol{G^r}$ described above
using the minimal monomial generators of $I^r$
(Section~\cref{ss:homogenize}) and we describe explicitly the
differential of the homogenized cellular chain complex that is
supported on $\ol{G^r}$; see \eqref{e:res-again}. We then show in
\cref{p:isomcomplexes} that this complex is isomorphic to a strand of
the Koszul complex resolving the Rees algebra of $I$. The fact that
this chain complex is a minimal free resolution of $I^r$ is a
consequence of the fact that the ideal $I$ is of linear type and its
Rees algebra is a complete intersection, as shown in
\cref{t:lineartype,t:regular}. In particular, we find explicit
formulas for the projective dimension and the Betti numbers of $I^r$;
see \cref{c:pd-I^r,c:betti-I^r}.

Our construction of the powers $\ol{G^r}$ points towards the
possibility of defining, more generally, the powers of any simplicial
(or cell) complex, and providing additional classes where
\cref{q:driving} has a positive answer. This is a topic for ongoing and future work.

The interested reader might be curious about a slightly different, but related
  version of \cref{q:driving}: if $\Gamma$ is a $q$-simplex (which
  supports a free resolution of {\emph any} ideal generated by $q$
  monomials) and $r$ is a positive integer, can we construct a cell
  complex $\Gamma^r$, starting from $\Gamma$, which supports a free
  resolution of $I^r$, where $I$ is {\emph any} ideal generated by $q$
  monomials? This question has been addressed in \cite{L2,Lr}.

  \section{Setup}\label{s:setup} 

This section provides the background and notation that will be used
throughout the rest of the paper, by building a correspondence between
monomial ideals and combinatorial structures that support their
resolutions.

\begin{notation} If $G$ is a graph with vertices $V(G)=\{x_1, \ldots, x_n\}$, then an undirected edge between vertices $x_j$ and $x_i$ will be denoted $\{x_j, x_i\}$ while a directed edge from $x_j$ to $x_i$ will be written $[x_j, x_i]$. The graphs used in this work will be simple graphs, that is, without loops or multiple edges. Throughout the paper, all graphs will be assumed to be connected unless otherwise indicated.
\end{notation}

 \begin{bfchunk}{Cell complexes.}\label{ss:cellcomplexes} The topological objects in this paper are {\emph polyhedral cell
  complexes}. See \cite{M80, OW} for additional resources on these
topics.
 \end{bfchunk}

\begin{definition} (\cite[p.~62]{MS})\label{d:pcx} Let $X$ be a
  finite collection of convex polytopes in a real vector space
  $\mathbb R^q$.  If these convex polytopes, called {\bf faces} of
  $X$, satisfy the two properties below, then $X$ is said to be  a {\bf
    polyhedral (or polytopal) cell complex}:
\begin{enumerate}
\item if $P$ is a polytope in $X$ and $F$ is a face of $P$, then $F$
  is in $X$;
\item if $P$ and $Q$ are polytopes in $X$, then $P \cap Q$ is a face
  of $P$ and a face of $Q$.
\end{enumerate}
\end{definition}

The faces of the polyhedral cell complexes that we will see in this
paper will be cubes of varying dimensions. Specifically, our
  $n$-cells will always be $n$-dimensional cubes.

\begin{definition}\label{d:cube}
An {\bf $n$-cube} $C_n$ is the cartesian product of $n$ unit
intervals. That is, $C_n = I_1 \times \cdots \times I_n$, for $n \geq 1$, where $I_i$
is a unit interval $[0,1]$.  The {\bf boundary} of the $n$-cube consists of
the $n-1$ cubes formed by replacing one of the unit intervals by one
of its two boundary points, which are $0$-cubes. Thus, there are $2n$ boundary components,
each of which has the form of $I_1 \times \cdots \times \{0\} \times
\cdots \times I_n$ or $I_1 \times \cdots \times \{1\} \times \cdots
\times I_n$. 

An $n$-cube can be built in $\RR^q$ when
$q\geq n$ by taking a point $\ba$ in $\RR^q$ and a collection of
$n$ standard unit vectors $\be_{i_1}, \ldots, \be_{i_n}$. The vertices of the cube are the endpoints of the vectors
$$\ba+ \sum_{j \in A} \be_{i_j}$$ for all $A \subseteq [n] = \{1,\ldots,n\}$.  The edges of the cube, each of which has the form 
$$\{\ba+ \sum_{j \in A} \be_{i_j}, \ba+ \sum_{j \in A} \be_{i_j} + \be_{i_k}\} \qforsome k \not\in A,$$
inherit a natural direction from that of $\be_{i_k}$. Viewing unit vectors as embedded copies of the unit interval $[0,1]$, directed from $0$ to $1$, the directed edges are then
$$[\ba+ \sum_{j \in A} \be_{i_j}, \ba+ \sum_{j \in A} \be_{i_j} + \be_{i_k}] \qforsome k \not\in A.$$
A vertex $v$ of a directed graph is
a {\bf sink} if every edge that contains $v$ is of the form
$[w,v]$ for some vertex $w$. Similarly, $v$ is a {\bf source} if
every edge that contains $v$ is of the form $[v,w]$ for some
vertex $w$. Note that 
$\ba$ is a source of the directed graph formed by the edges of the cube described above and $ \bb = \ba+ \sum _{j \in [n]} \be_{i_j}$ is a sink. 

\end{definition}

Given a polyhedral cell complex $\Gamma$ let $\Gamma^{(n)}$ be the set of
$n$-cells of $\Gamma$.  For convenience, $\varnothing$ is considered to
be a $(-1)$-cell, and $\Gamma^{(-1)}=\{\varnothing\}$.

\begin{definition}
\label{oriented-chain-complex}
For  a field $\sf k$ and each $i\ge 0$,
let $\sfk^i$ denote an $i$-dimensional  $\sf k$-vector space.
The {\bf oriented chain complex} of $\Gamma$ is the complex
$$
C(\Gamma, \sfk): \qquad \dots\to \sfk^{|\Gamma^{(i)}|}\xrightarrow{\partial_i} \sfk^{|\Gamma^{(i-1)}|}\to \dots \to \sfk^{|\Gamma^{(1)}|}\xrightarrow{\partial_1}\sfk^{|\Gamma^{(0)}|}
$$ with differentials defined as follows:  For $c \in \Gamma^{(i)}$ and $c'\in \Gamma^{(i-1)}$, let $\varepsilon(c,c')=0$ if $c'$ is not a face of $c$, and otherwise $\varepsilon(c,c')= \pm 1$, chosen by convention, commonly by using an orientation or an incidence function, so that $\partial^2=0$.
Then for all $i\ge 1$, $c\in \Gamma^{(i)}$, and
basis elements $\bu_c$, define
$$
\partial_i(\bu_c)=\sum_{c'\in \Gamma^{(i-1)}}\varepsilon(c,c')\bu_{c'} \,.
$$
\end{definition}

In particular, if $\Gamma$ is a polyhedral cell complex (\cref{d:pcx}), the faces of $\Gamma$ can be
oriented (in an arbitrary manner) so that the boundary chain of a face
$F$ is $\partial(F) = \sum \varepsilon (F, G) G$, where the sum is
taken over all maximal proper faces $G$ of $F$ and
$$
\varepsilon (F,G) = \begin{cases} +1 & {\text{ if the orientation of }} F {\text{ induces the orientation of }} G; \\
-1 & {\text{ otherwise. }}
\end{cases}
$$

\begin{bfchunk}{Cellular Resolutions.}
\label{s:homogenization} Throughout, assume that $R=\sfk[x_1, \ldots, x_n]$ is a polynomial ring over a
field $\sfk$ and $I=(m_0,\ldots,m_q)$ is an ideal generated by
monomials. A {\bf graded free resolution of $I$} is an exact
sequence of free $S$-modules of the form:
\begin{equation}
\label{resolution}
\mathbb{F}: \ \ \ 0 \to M_d \stackrel{\partial_d}{\longrightarrow} \cdots \to M_i
\stackrel{\partial_i}{\longrightarrow} M_{i-1} \to\cdots \to M_1
\stackrel{\partial_1}{\longrightarrow} M_0
\end{equation}
 where $I \cong
M_0/\im(\partial_1)$, and each map $\partial_i$ is graded, in the
sense that it preserves the degrees of homogeneous elements.

If $\partial_i(M_i) \subseteq (x_1,\ldots,x_n) M_{i-1}$ for every
$i>0$, then the free resolution $\mathbb{F}$ is {\bf minimal}.
The length of a minimal free resolution of $I$ (which is $d$ in the case of
$\mathbb{F}$ above) is another invariant of $I$ called the {\bf
  projective dimension} and is denoted by $\pd_R(I)$. 

One concrete way to calculate a multigraded free resolution is to use
chain complexes of topological objects, and in particular of cellular
chain complexes. This approach was initiated by Taylor~\cite{T}, and
further developed by Bayer and Sturmfels~\cite{BS} and many other
researchers.

Let $\Gamma$ be a polyhedral cell complex, or more generally any regular CW complex, with vertex set labeled by the monomials $m_0, \dots, m_q$. We label each cell $c\in \Gamma$ by  $\lcm(c)$, which is defined as the least common multiple of the labels of its   vertices. The {\bf homogenization} of the oriented chain complex $C(\Gamma, \sfk)$ defined in \cref{oriented-chain-complex}  is a complex $\mathbb F=\mathbb F_\Gamma$ as displayed in \eqref{resolution}, such that 
$$M_i=\bigoplus_{c\in \Gamma^{(i)}}R(\lcm(c))\,.$$ We denote the basis
element corresponding to  the free module $R(\lcm(c))$ in this
sum by $\bu_c$. The differential of $\mathbb F_\Gamma$ is described by
$$
\partial_i(\bu_c)=\sum_{c'\in \Gamma^{(i-1)}}\varepsilon(c,c')\frac{\lcm(c)}{\lcm(c')}\bu_{c'} 
$$
for each $c\in \Gamma^{(i)}$. Recall that $\varepsilon(c,c')$ is nonzero only when $c'$ is a face of $c$, and in this case $\lcm(c')$ divides $\lcm(c)$. 

We say that $\Gamma$ {\bf supports a resolution} of $I= (m_0, \dots,
m_q)$ if the chain complex $\mathbb F_\Gamma$ is a resolution of
$I$. In this case, we say $\mathbb F_\Gamma$ is a {\bf cellular
  resolution} of $I$.

Note that if $\mathbb F_\Gamma$ is a resolution of $I$, then it is minimal if and only if $\lcm(c)\ne \lcm(c')$ for every cell $c$ and every maximal face $c'$ of $c$, see for example \cite{BS} or \cite{P}.
\end{bfchunk}

\begin{bfchunk}{Ideals of projective dimension one.}
\label{ss:pd1}  In \cite[Theorem 27]{FH}, Faridi and Hersey proved that for a monomial
ideal $I$, having a minimal free resolution supported on a tree is
equivalent to having projective dimension one. Here a graph can be
viewed as a polyedral complex where the maximum dimension of a cell is
one.  Moreover, Faridi and
Hersey gave a concrete construction describing how to build the tree
given a monomial generating set of an ideal of projective dimension
one.
\end{bfchunk}

\begin{example}\label{e:running1} Let $I=(xy,yz,zu)$ in $R = \sfk[x,y,z,u]$. Then 
  $\pd_R(I)=1$, and $I$ has a minimal resolution supported on
  the labeled graph below.

{\small 
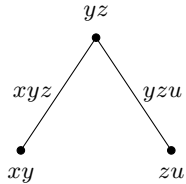
\begin{center}
		\begin{tikzpicture}
			\coordinate (A) at (0, 0);
			\coordinate (B) at (2,0);
			\coordinate (C) at (1, 1.5);
			\coordinate (D) at (0.5, 0.75);
			\coordinate (E) at (0.5, 0.75);
			\coordinate (F) at (1.5, 0.75);
			\draw[black, fill=black] (A) circle(0.05);
			\draw[black, fill=black] (B) circle(0.05);
			\draw[black, fill=black] (C) circle(0.05);
			
			\draw[-] (A) -- (C);
			\draw[-] (B) -- (C);
			
			\node[label = below :$xy$] at (A) {};
			\node[label = below :$zu$] at (B) {};
			\node[label = above :$yz$] at (C) {};
			\node[label ={[label distance=-4pt] left :$xyz$}] at (D) {};
			\node[label ={[label distance=-3pt] right :$yzu$}] at (F) {};
		\end{tikzpicture}
\end{center}} 
\end{example}

\section{Powers of Trees}\label{s:graph product}

In this section we show that if $G$ is a graph supporting a minimal free
  resolution of a square-free monomial ideal $I$ and
  $r$ is a positive integer, then we can build a polyhedral cell complex
  $\ol{G^r}$ from the $r^{th}$ power graph $G^r$ (described below), which
   supports a minimal free resolution of 
  $I^r$. In this set, the ideal $I$ has projective dimension one in $R = \sfk[x_1, \dots, x_n]$,
  as per \cite[Theorem 27]{FH}.

  Our definition of $G^r$ (\cref{d:graph-powers}) is an abelianized
  extension of a well-known construction in graph theory called the
  {\emph box product}.  Given two graphs $G$ and $H$, the
  {\bf Cartesian}, or {\bf box}, {\bf product} of $G$ and $H$, denoted by $G
  \square H$, is a new graph whose vertex set is the
  Cartesian product of the vertices of $G$ with those of $H$, and whose edges are
  obtained from the edges of $G$ and $H$.  If $G$ and $H$ have the
  same vertex set, then one can define an abelian version of this
  product by forming a graph quotient that identifies $(v_i, v_j)$
  with $(v_j, v_i)$. Note that no loops or double edges are created in
  this process since if either $\{(v_i,v_j),(v_j, v_i)\}$ is an edge of $G \square H$
  or $(v_i,v_j)$ and $(v_j, v_i)$ have a common neighboring edge, then $i=j$.

\begin{definition}[{\bf The (directed) graph $G^r$}]\label{d:graph-powers}
   Let $G$ be a graph on the vertex set $\{ v_0, \ldots, v_{q}\}$ and let $r$ be a
   positive integer. Define
   $$\cNr=\{(a_0,\ldots,a_q) \in \ZZ_{\geq0}^{q+1} \mid
 a_0+ \cdots+a_q=r\}.$$ Let $G^r$ be the graph with distinct vertices labeled $ \bva$ for each $ \ba
 \in \cNr$, that is
 $$V(G^r)=\{ \bva=v_0^{a_0} v_1^{a_1}\cdots v_{q}^{a_{q}} \mid \ba=(a_0,\ldots,a_q)
 \in \cNr \},$$ and edge set 
 $$E(G^r)=\left\{ \{\bva, \bvb \} \mid   \bva=Wv_j, \bvb = Wv_i
 \mbox{ for some }\{v_j, v_i\} \in E(G)  \mbox{ and } W \in V(G^{r-1}) \right\},$$
  where if $W= \bvc$ then $Wv_i= {\bf v^{c+f_i}}$. Here $\bff_0, \dots, \bff_q$ denotes the standard basis of $\mathbb R^{q+1}$; where the indexing starts at $0$ for later convenience. More precisely, $\bff_i$ denotes the $(i+1)^{st}$ standard basis vector  \begin{equation}\label{e:f} \bff_i=(f_0, \dots, f_q), \qwith  f_i= 1 \qand  f_k=0 \qif k \neq i. 
   \end{equation}
  
 If $G$ is a directed graph, an edge $\{\bva, \bvb\}$ of $G^r$ as
 described above inherits its direction from that of $\{v_j, v_i\}$.
We denote the directed edge from $\bva$ to $\bvb$ by $[\bva, \bvb]$.
\end{definition}

Notice that under this definition, if $G$ is a directed graph then 
$$[v_0^{a_0}v_1^{a_1} \cdots v_{q}^{a_{q}}, \,
 v_0^{b_0}v_1^{b_1} \cdots v_{q}^{b_{q}}] \in E(G^{r})$$ if and only if there exits
 an edge $[v_j, v_i]$  of $G$ with $$a_j +1 =b_j, a_i-1=b_i \qand 
 a_k=b_k \qfor k\neq i,j.$$
 
\begin{example} \label{e:boxproduct}
From the path $G$ below, we form the product $G^2$ by gluing the paths $G_x^2$, $G_y^2$, $G_z^2$, and $G_w^2$  which are obtained by multiplying the vertices in the path $G$ by $v_0=x, v_1=y, v_2=z$ and $v_3=w$ respectively.  Note that $G_x^2$ and $G_w^2$, which appear in $G^2$ along the top and the right, respectively, are glued at $xw$; $G_y^2$ and $G_z^2$ are glued at $yz$.

{\small 
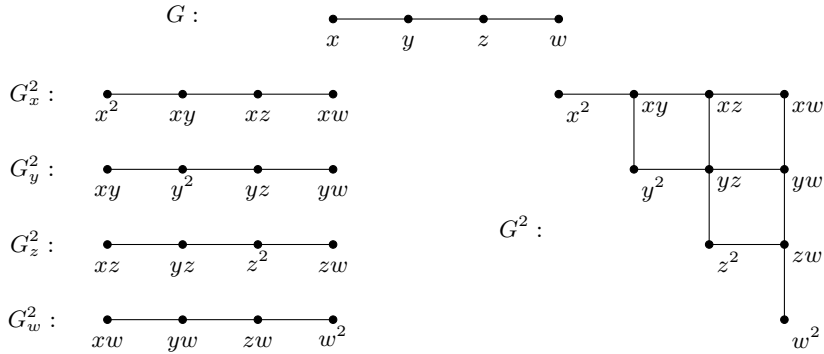
\begin{tikzpicture}
\coordinate (Z) at (1, 1.4);
\coordinate (A) at (3, 1);
\coordinate (B) at (4,1);
\coordinate (C) at (5, 1);
\coordinate (D) at (6, 1);
\draw[black, fill=black] (A) circle(0.05);
\draw[black, fill=black] (B) circle(0.05);
\draw[black, fill=black] (C) circle(0.05);
\draw[black, fill=black] (D) circle(0.05);
\draw[-] (A) -- (B);
\draw[-] (B) -- (C);
\draw[-] (C) -- (D);
\node[label = below :$x$] at (A) {};
\node[label = below :$y$] at (B) {};
\node[label = below :$z$] at (C) {};
\node[label = below :$w$] at (D) {};
\node[label = below :$G:$] at (Z) {};

\\
\coordinate (Z1) at (-1, 0.4);
\coordinate (A1) at (0, 0);
\coordinate (B1) at (1,0);
\coordinate (C1) at (2, 0);
\coordinate (D1) at (3, 0);
\draw[black, fill=black] (A1) circle(0.05);
\draw[black, fill=black] (B1) circle(0.05);
\draw[black, fill=black] (C1) circle(0.05);
\draw[black, fill=black] (D1) circle(0.05);
\draw[-] (A1) -- (B1);
\draw[-] (B1) -- (C1);
\draw[-] (C1) -- (D1);
\node[label ={[label distance=-4pt] below :$x^2$}] at (A1) {};
\node[label = below :$xy$] at (B1) {};
\node[label = below :$xz$] at (C1) {};
\node[label = below :$xw$] at (D1) {};
\node[label = below :$G^2_x:$] at (Z1) {};

\coordinate (Z2) at (-1, -0.6);
\coordinate (A2) at (0, -1);
\coordinate (B2) at (1,-1);
\coordinate (C2) at (2, -1);
\coordinate (D2) at (3, -1);
\draw[black, fill=black] (A2) circle(0.05);
\draw[black, fill=black] (B2) circle(0.05);
\draw[black, fill=black] (C2) circle(0.05);
\draw[black, fill=black] (D2) circle(0.05);
\draw[-] (A2) -- (B2);
\draw[-] (B2) -- (C2);
\draw[-] (C2) -- (D2);
\node[label ={[label distance=-4pt] below :$y^2$}] at (B2) {};
\node[label = below :$xy$] at (A2) {};
\node[label = below :$yz$] at (C2) {};
\node[label = below :$yw$] at (D2) {};
\node[label = below :$G^2_y:$] at (Z2) {};

\coordinate (Z3) at (-1, -1.6);
\coordinate (A3) at (0, -2);
\coordinate (B3) at (1,-2);
\coordinate (C3) at (2, -2);
\coordinate (D3) at (3, -2);
\draw[black, fill=black] (A3) circle(0.05);
\draw[black, fill=black] (B3) circle(0.05);
\draw[black, fill=black] (C3) circle(0.05);
\draw[black, fill=black] (D3) circle(0.05);
\draw[-] (A3) -- (B3);
\draw[-] (B3) -- (C3);
\draw[-] (C3) -- (D3);
\node[label ={[label distance=-4pt] below :$z^2$}] at (C3) {};
\node[label = below :$xz$] at (A3) {};
\node[label = below :$yz$] at (B3) {};
\node[label = below :$zw$] at (D3) {};
\node[label = below :$G^2_z:$] at (Z3) {};

\coordinate (Z4) at (-1, -2.6);
\coordinate (A4) at (0, -3);
\coordinate (B4) at (1,-3);
\coordinate (C4) at (2, -3);
\coordinate (D4) at (3, -3);
\draw[black, fill=black] (A4) circle(0.05);
\draw[black, fill=black] (B4) circle(0.05);
\draw[black, fill=black] (C4) circle(0.05);
\draw[black, fill=black] (D4) circle(0.05);
\draw[-] (A4) -- (B4);
\draw[-] (B4) -- (C4);
\draw[-] (C4) -- (D4);
\node[label ={[label distance=-4pt] below :$w^2$}] at (D4) {};
\node[label = below :$xw$] at (A4) {};
\node[label = below :$yw$] at (B4) {};
\node[label = below :$zw$] at (C4) {};
\node[label = below :$G^2_w:$] at (Z4) {};

\coordinate (W) at (5.5, -1.4);
\coordinate (K) at (6, 0);
\coordinate (L) at (7,0 );
\coordinate (M) at (8,0);
\coordinate (N) at (9, 0);
\coordinate (E) at (7, -1);
\coordinate (F) at (8, -1);
\coordinate (G) at (9, -1);
\coordinate (H) at (8,-2);
\coordinate (I) at (9, -2);
\coordinate (J) at (9, -3);
\draw[black, fill=black] (K) circle(0.05);
\draw[black, fill=black] (L) circle(0.05);
\draw[black, fill=black] (M) circle(0.05);
\draw[black, fill=black] (N) circle(0.05);
\draw[black, fill=black] (E) circle(0.05);
\draw[black, fill=black] (F) circle(0.05);
\draw[black, fill=black] (G) circle(0.05);
\draw[black, fill=black] (H) circle(0.05);
\draw[black, fill=black] (I) circle(0.05);
\draw[black, fill=black] (J) circle(0.05);
\draw[-] (K) -- (L);
\draw[-] (L) -- (M);
\draw[-] (M) -- (N);
\draw[-] (L) -- (E);
\draw[-] (M) -- (F);
\draw[-] (N) -- (G);
\draw[-] (E) -- (F);
\draw[-] (F) -- (G);
\draw[-] (F) -- (H);
\draw[-] (G) -- (I);
\draw[-] (H) -- (I);
\draw[-] (I) -- (J);
\node[label = {[label distance=-5pt] below right :$x^2$}] at (K) {};
\node[label =  {[label distance=-5pt] below right :$xy$}]  at (L) {};
\node[label = {[label distance=-5pt] below right:$xz$}] at (M) {};
\node[label = {[label distance=-5pt] below right :$xw$}] at (N) {};
\node[label = {[label distance=-5pt] below right :$y^2$}] at (E) {};
\node[label = {[label distance=-5pt] below right :$yz$}] at (F) {};
\node[label = {[label distance=-5pt] below right :$yw$}] at (G) {};
\node[label = {[label distance=-5pt] below right :$z^2$}] at (H) {};
\node[label = {[label distance=-5pt] below right :$zw$}] at (I) {};
\node[label = {[label distance=-5pt] below right :$w^2$}] at (J) {};
\node[label = below :$G^2:$] at (W) {};
\end{tikzpicture} } 

\end{example}

\begin{construction}[{\bf Labeling and directing a rooted tree}]\label{n:direction} 
Let $G$ be a tree with $q+1$ vertices. Fix a vertex $v_0$ of $G$ to be
the {\bf root} of the tree. Label the remaining vertices so that
the vertices along the unique path from $v_0$ to any vertex are
labeled in increasing order, so that if $$v_0, \, v_{i_1}, \, v_{i_2},
\ldots , v_{i_t}$$ are the distinct vertices of a path between $v_0$
and $v_{i_t}$, then $i_j < i_k$ whenever $1\leq j <k \leq t$.  With
this labeling, a direction on the edges of $G$ is defined by writing
every edge $[v_j,v_i]$ as an ordered pair $e_i=[v_j, v_i]$ where $j <
i$.

\smallskip

\noindent \pmb{$\tau(i)$:} We denote the index $j$ in
the directed edge $e_i$ by $\tau(i)$, so that the
directed edges of $G$ can be written as $$e_i=[v_{\tau(i)}, \, v_i] \qfor i \in
\{1,\ldots,q\}.$$
 \end{construction}

For the remainder of the paper, all directed graphs will be assumed to have vertices labeled in accordance with \cref{n:direction}. In particular, $v_0$ will always denote the root of a directed tree.

Notice that this uniquely labels the $q$ edges of $G$ by $e_1, \ldots
e_q$, with each $e_i$ directed toward $v_i$ as seen in
\cref{directedboxproduct} below.

Furthermore, the edges in $G^r$ inherit their direction from the edges
of $G$: if $[v_{\tau(i)}, v_i]$ is a directed edge of $G$, then the
corresponding edge $[\bva v_{\tau(i)}, \bva v_i]$ is a directed edge
of $G^r$ for every $\ba \in \cN_{r-1}$.

 \begin {example} \label{directedboxproduct} We direct the edges of $G$ and $G^2$ of \cref{e:boxproduct} by picking $v_0=x$ as the root 
 
{\small  
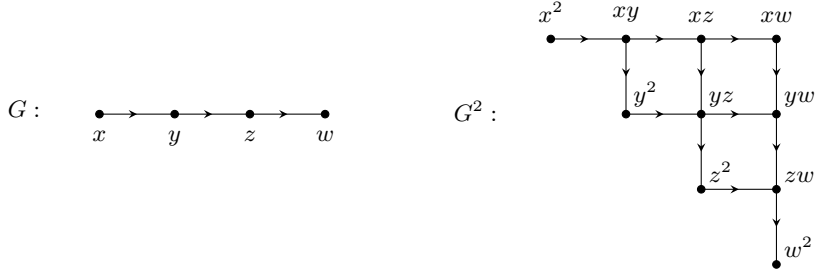
\begin{tikzpicture} [decoration={markings, 
	mark= at position 0.5 with {\arrow{stealth}}, } ]

\coordinate (Z) at (-1, 0.4);
\coordinate (A) at (0, 0);
\coordinate (B) at (1,0);
\coordinate (C) at (2, 0);
\coordinate (D) at (3, 0);
\draw[black, fill=black] (A) circle(0.05);
\draw[black, fill=black] (B) circle(0.05);
\draw[black, fill=black] (C) circle(0.05);
\draw[black, fill=black] (D) circle(0.05);
\draw[postaction={decorate}] (A) -- (B);
\draw[postaction={decorate}] (B) -- (C);
\draw [postaction={decorate}](C) -- (D);
\node[label = below :$x$] at (A) {};
\node[label = below :$y$] at (B) {};
\node[label = below :$z$] at (C) {};
\node[label = below :$w$] at (D) {};
\node[label = below :$G:$] at (Z) {};

\coordinate (W) at (5, 0.4);
\coordinate (K) at (6, 1);
\coordinate (L) at (7,1 );
\coordinate (M) at (8,1);
\coordinate (N) at (9, 1);
\coordinate (E) at (7, 0);
\coordinate (F) at (8, 0);
\coordinate (G) at (9, 0);
\coordinate (H) at (8,-1);
\coordinate (I) at (9, -1);
\coordinate (J) at (9, -2);
\draw[black, fill=black] (K) circle(0.05);
\draw[black, fill=black] (L) circle(0.05);
\draw[black, fill=black] (M) circle(0.05);
\draw[black, fill=black] (N) circle(0.05);
\draw[black, fill=black] (E) circle(0.05);
\draw[black, fill=black] (F) circle(0.05);
\draw[black, fill=black] (G) circle(0.05);
\draw[black, fill=black] (H) circle(0.05);
\draw[black, fill=black] (I) circle(0.05);
\draw[black, fill=black] (J) circle(0.05);
\draw[postaction={decorate}] (K) -- (L);
\draw[postaction={decorate}] (L) -- (M);
\draw[postaction={decorate}] (M) -- (N);
\draw[postaction={decorate}] (L) -- (E);
\draw[postaction={decorate}] (M) -- (F);
\draw[postaction={decorate}] (N) -- (G);
\draw[postaction={decorate}] (E) -- (F);
\draw[postaction={decorate}] (F) -- (G);
\draw [postaction={decorate}](F) -- (H);
\draw[postaction={decorate}] (G) -- (I);
\draw[postaction={decorate}] (H) -- (I);
\draw[postaction={decorate}] (I) -- (J);
\node[label = above  :$x^2$] at (K) {};
\node[label =  above  :$xy$]  at (L) {};
\node[label = above :$xz$] at (M) {};
\node[label = above :$xw$] at (N) {};
\node[label = {[label distance=-5pt] above right :$y^2$}] at (E) {};
\node[label = {[label distance=-5pt] above right :$yz$}] at (F) {};
\node[label = {[label distance=-5pt] above right :$yw$}] at (G) {};
\node[label = {[label distance=-5pt] above right :$z^2$}] at (H) {};
\node[label = {[label distance=-5pt] above right :$zw$}] at (I) {};
\node[label = {[label distance=-5pt] above right :$w^2$}] at (J) {};
\node[label = below :$G^2:$] at (W) {};
\end{tikzpicture}  }

 \end{example}
 
\begin{lemma}\label{l:j-unique} Let $G$ be a directed  tree  on 
    vertices $v_0,\ldots,v_q$, labeled as in \cref{n:direction}, let
    $r,q>0$, and let $\ba=(a_0,\ldots,a_q),$ $\bb=(b_0,\ldots,b_q) \in
    \cNr$. Then $[\bva, \bvb]$ is a directed edge of $G^r$ if and only
    if for a unique $i \in [q]$, we have
    $$\ba=\bb-\bff_i + \bff_{\tau(i)}$$
    where $\bff_j$ denotes the $(j+1)^{st}$ standard basis vector in $\RR^{q+1}$ as in \eqref{e:f}.
  \end{lemma}

   \begin{proof}  By definition,  $[\bva, \bvb]$ is a directed edge  of
     $G^r$ if and only if for some $i \in [q]$ and $\bc
     =(c_0,\ldots,c_q) \in \cN_{r-1}$,
     $$\bva=\bv^{\bc} \cdot v_{\tau(i)} \qand \bvb=\bv^{\bc} \cdot
     v_i.$$ This happens if and only if  $a_j=b_j=c_j$ when $j \notin
     \{i,\tau(i)\}$, and since $\tau(i) \neq i$, $b_i=c_i+1=a_i+1$ and
     $b_{\tau(i)}+1=c_{\tau(i)}+1=a_{\tau(i)}$. Thus  $$\ba=\bb-\bff_i + \bff_{\tau(i)}.$$

      The uniqueness of $i$
     also follows from the same observation, since the only
     coordinates in which $\ba$ and $\bb$ differ are $i$ and $\tau(i)$, and we
     know $\tau(i) < i$, $b_i=a_i+1$, and $b_{\tau(i)}+1=a_{\tau(i)}$,
     so the roles of $i$ and $\tau(i)$ cannot be reversed.
   \end{proof}

\begin{bfchunk}{Embedding $G^r$ in  $\RR^q$ as $1$-skeleta of cubes.}\label{ss:embed}
We now define an explicit embedding of $G^r$ into the Euclidean space
$\mathbb R^q$ when $G$ is a tree. The embedding chosen is based on the edges of $G$ and
designed so that each edge of $G^r$ will be parallel to an axis of
$\mathbb R^q$. 
\end{bfchunk}

\begin{definition}\label{matrix}
Let $G$ be a directed rooted tree on $q+1$ vertices $v_0,\ldots,v_q$,
labeled as in \cref{n:direction}. Define a $q \times q$ matrix
$\Phi=\Phi(G)=(\Phi_{i,j})$ by:
$$\Phi_{i,j}= \left\{\begin{array}{cc} 1& \text{ if } e_i \text{ lies
  on the unique path from } v_0 \text{ to } v_j, \\ 0 & \text{
  otherwise.} \end{array} \right.$$
\end{definition}

That is, $\Phi$ is the vertex-path incidence matrix whose rows are
indexed by edges $e_1, \ldots, e_q$ and whose columns are indexed by
the vertices $v_1, \ldots, v_q$ with $i,j$ entry indicating whether or
not the $e_i$  is in the path from  $v_0$ to $v_j$ in the
  graph $G$.  Using this matrix, we can embed the vertices of $G^r$
into $\mathbb R^q$.

\begin{example}\label{matrixphi} Let $G$ be the path graph in
  \cref{e:running1} with labels $v_0,v_1,v_2$
replacing the vertex labels $xy,yz,zu$, respectively.
Since $e_1$ lies on the unique path from $v_0$ to $v_1$ and $v_2$, but $e_2$ only lies on the unique path to $v_2$, the matrix $\Phi(G)$ is 
$$\Phi(G)=
\begin{pmatrix}
  1&1\\
  0&1
\end{pmatrix}.
$$
\end{example}

\begin{definition}\label{d:phir}
 Let $G$ be a directed rooted tree on $q+1$ vertices
  $v_0,\ldots,v_q$, labeled as in \cref{n:direction}. Define
$\varphi\colon \ZZ_{\geq 0}^{q+1} \to \mathbb R^{q}$ by
$$\varphi(a_0,a_1,\ldots,a_q)=
\Phi(G)\begin{pmatrix} a_1\\\vdots\\a_q
\end{pmatrix}. $$
For convenience, we write $\varphi(v_0^{a_0}v_1^{a_1}\cdots v_q^{a_q})=\varphi(a_0,\ldots,a_q) = \varphi(\ba)$.
  \end{definition}

\begin{lemma}\label{l:injective}
The function $\varphi_{\big |_{\cNr}}$ is injective. 
\end{lemma}

\begin{proof} First notice that if $(a_0,\ldots,a_q) \in \cNr$, then
    $a_0=r-(a_1+ \cdots+a_q)$ and so the value of $a_0$ is uniquely
    determined by the vector $(a_1,\ldots,a_q)$.

  By \cref{n:direction}, every vertex $v_j$ appearing on the 
unique path from $v_0$ to $v_i$ satisfies $j<i$. Thus, by the
labeling of the edges using their terminal vertex, the matrix $\Phi$
is an upper triangular matrix. In addition, the diagonal entries are
all equal to $1$ since by definition, $e_i$ is the final edge in the
unique path from $v_0$ to $v_i$. Therefore $\Phi(G)$ is a nonsingular
matrix, and as a result   $\varphi_{\big |_{\cNr}}$ is injective.
\end{proof}

Notice that under the embedding $\varphi$, the $i^{th}$ coordinate of the point $\varphi(v_0^{a_0}\cdots v_q^{a_q})$ in $\RR^q$ is the total number of times,
with multiplicity, that the edge $e_i$ appears on the path from $v_0$ to any
vertex $v_k$ with $a_k \geq 1$ and $k \geq 1$.

  \begin{lemma}\label{l:differbyj} Let $G$ be a tree, labeled as in \cref{n:direction}, $r,q>0$, and
    $\ba,\bb \in \cNr$. If $[\bva, \bvb]$ is a directed edge of $G^r$,
    and $i$ is as in \cref{l:j-unique}, then
    $$\varphi(\bb)=\varphi(\ba)+\be_i,$$ where $\be_i$ denotes the
      $i^{th}$ standard basis vector in $\RR^q$.
  \end{lemma}

   \begin{proof} Let $\ba=(a_0,\ldots,a_q)$ and  $\bb=(b_0,\ldots,b_q)$.
   If $[\bva, \bvb]$ is a directed edge of $G^r$, then using the
   unique $i$ from \cref{l:j-unique}, the unique path in $G$ from
   $v_0$ to $v_i$ is an extension of the unique path from $v_0$ to
   $v_{\tau(i)}$ by the edge $e_i$.

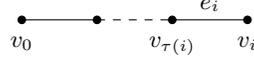
   \begin{center}
 {\small   \begin{tikzpicture}
\coordinate (A) at (3, 1);
\coordinate (B) at (4,1);
\coordinate (C) at (5, 1);
\coordinate (D) at (6, 1);
\draw[black, fill=black] (A) circle(0.05);
\draw[black, fill=black] (B) circle(0.05);
\draw[black, fill=black] (C) circle(0.05);
\draw[black, fill=black] (D) circle(0.05);
\draw[-] (A) -- (B);
\draw[dashed] (B) -- (C);
\draw[-] (C) -- node[above]{$e_i$} (D);
\node[label = below :$v_0$] at (A) {};
\node[label = below :$v_{\tau(i)}$] at (C) {};
\node[label = below :$v_i$] at (D) {};
   \end{tikzpicture} }
   \end{center}
   Therefore,  for $i \in [q]$,  if $\Phi(G)_i$ denotes the $i^{th}$ column of the matrix, 
   $$\Phi(G)_i= \left \{ \begin{array}{cl}
     \be_i & \qif \tau(i)=0 \\
    \Phi(G)_{\tau(i)} + \be_i & \qif \tau(i) \geq 1.
\end{array} \right.$$
  By \cref{l:j-unique}
  we have
   \begin{align*} \varphi(\bb)=&\Phi(G)\begin{pmatrix}
 		b_1 &\cdots &b_q \end{pmatrix}^T \\
 	=&\Phi(G)\begin{pmatrix}
 		a_1&\cdots &a_q \end{pmatrix}^T - \Phi(G)_{\tau(i)} + 
 \Phi(G)_i\\
 	= &\varphi(\ba)+ \be_i,
 \end{align*} where when  $\tau(i)=0$, we set $\Phi(G)_0 = {\bf 0}$ for convenience.
      \end{proof}

   \begin{definition}[{\bf The directed graph $\varphi(G^r)$}]
     Let $G$ be a directed rooted tree as in
       \cref{n:direction}. We define $\varphi(G^r)$ to be the
     directed graph embedded in $\RR^q$ whose vertices are
     $\varphi(\ba)$ where $\ba \in \cNr$ and whose edges are
     $[\varphi(\ba), \varphi(\bb)]$ where $[\bva, \bvb]$ is an edge of
     $G^r$. By \cref{l:differbyj} we must have
     $\varphi(\bb)=\varphi(\ba)+ \be_i$ for a unique $i \in
     [q]$. \end{definition}

   To summarize we observe the following correspondence between edges
   of the directed graphs $G^r$ and $\varphi(G^r)$. If $\ba,\bb \in \cNr$, then
 for a unique  $j \in [q]$
   \begin{equation} \label{e:edges}
     \begin{split}
     [\varphi(\ba), \varphi(\bb)] \in E(\varphi(G^r))&
     \iff  \varphi(\bb)=\varphi(\ba)+ \be_i \\
     &\iff \bb= \ba +\bff_i - \bff_{\tau(i)} \\
     &\iff [\bva, \bvb] \in E(G^r)
     \end{split}
   \end{equation}
   where $\be_i$ and $\bff_i$ are unit 
   vectors in $\RR^q$ and $\RR^{q+1}$ respectively, as in \cref{d:graph-powers}.

\begin{example}\label{ex:path-0} 
Let $G$ be the graph in \cref{e:running1} with matrix $\Phi(G)=
\begin{pmatrix}
  1&1\\
  0&1
\end{pmatrix}$ computed in \cref{matrixphi}. The vertex set of $G^2$ is $\{v_0^2, v_0v_1, v_0v_2, v_1^2, v_1v_2, v_2^2\}$ with corresponding exponent vectors
  $$\{(2,0,0), (1,1,0), (1,0,1),(0,2,0), (0,1,1), (0,0,2)\}.$$ 
  We thus have 
  $$
  \varphi(2,0,0)=
  \begin{pmatrix}
  1&1\\
  0&1
\end{pmatrix}
  \begin{pmatrix} 0\\0 \end{pmatrix}
  =  \begin{pmatrix} 0\\0 \end{pmatrix}
\qand
  \varphi(0,1,1)=
  \begin{pmatrix}
  1&1\\
  0&1
\end{pmatrix}
  \begin{pmatrix} 1\\1 \end{pmatrix}
  =  \begin{pmatrix} 2\\1 \end{pmatrix}.
    $$
  Similarly we have
  $$\varphi(1,1,0)=\begin{pmatrix} 1\\ 0\end{pmatrix},\ 
\varphi(1,0,1)=\begin{pmatrix}1 \\1 \end{pmatrix},\ 
\varphi(0,2,0)=\begin{pmatrix}2 \\0 \end{pmatrix},\ 
\varphi(0,0,2)=\begin{pmatrix} 2\\2 \end{pmatrix}.$$

 The graph $\varphi(G^2)$ is below.

 
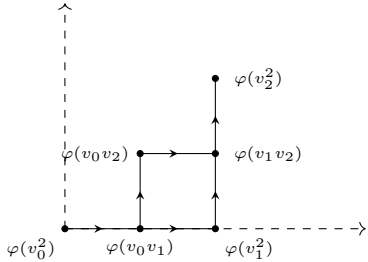
\begin{center}
	\begin{tikzpicture} [decoration={markings, 
			mark= at position 0.5 with {\arrow{stealth}},
		}
		] 	
		
		\coordinate (A) at (-3, 0);
		\coordinate (B) at (-2, 1);
		\coordinate (C) at (-1, 1);
		\coordinate (D) at (-2, 0);
		\coordinate (E) at (-1, 0);
		\coordinate (F) at (-1, 2);
		\coordinate (X) at (1,0);
		\coordinate (Y) at (-3,3);
		\draw[black, fill=black] (A) circle(0.04);
		\draw[black, fill=black] (B) circle(0.04);
		\draw[black, fill=black] (C) circle(0.04);
		\draw[black, fill=black] (D) circle(0.04);
		\draw[black, fill=black] (E) circle(0.04);
		\draw[black, fill=black] (F) circle(0.04);
		\draw[postaction={decorate}] (D) -- (B);
		\draw[postaction={decorate}](B) -- (C);
		\draw[postaction={decorate}] (E) -- (C);
		\draw[postaction={decorate}](A) -- (D);
		\draw [postaction={decorate}](D) -- (E);
		\draw[postaction={decorate}] (C) -- (F);
		\draw[->, dashed] (A) -- (X);
		\draw[->, dashed] (A) -- (Y);
		\node[label = {[label distance=-5pt] below left :\tiny $\varphi(v_0^2)$}] at (A) {};
		\node[label = {[label distance=-3pt]  below :\tiny $\varphi(v_0v_1)$}] at (D) {};
		\node[label = right :\tiny{$\varphi(v_1v_2)$}]at (C) {};
		\node[label = {[label distance=-3pt]  left :\tiny $\varphi(v_0v_2)$}] at (B) {};
		\node[label = {[label distance=-5pt]  below right :\tiny $\varphi(v_1^2)$}] at (E) {};
		\node [label = right :\tiny{$\varphi(v_2^2)$}] at (F) {};
	\end{tikzpicture}
\end{center}
\end{example}

\begin{example} 
\label{e:G^3}Let $G$ be the directed tree with $4$ vertices that
  starts at the root $v_0$ and consists of two distinct paths as shown in the figure below. The directed edges of $G$ have the form 
  $[v_j,v_i]$ with $j<i$.
  
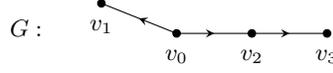
 $$ \small{  \begin{tikzpicture} [decoration={markings, 
			mark= at position 0.5 with {\arrow{stealth}}, } ]
		
		\coordinate (Z) at (-1, 0.4);
		\coordinate (A) at (0, 0.4);
		\coordinate (B) at (1,0);
		\coordinate (C) at (2, 0);
		\coordinate (D) at (3, 0);
		\draw[black, fill=black] (A) circle(0.05);
		\draw[black, fill=black] (B) circle(0.05);
		\draw[black, fill=black] (C) circle(0.05);
		\draw[black, fill=black] (D) circle(0.05);
		\draw[postaction={decorate}] (B) -- (A);
		\draw[postaction={decorate}] (B) -- (C);
		\draw [postaction={decorate}](C) -- (D);
		\node[label = below :$v_1$] at (A) {};
		\node[label = below :$v_0$] at (B) {};
		\node[label = below :$v_2$] at (C) {};
		\node[label = below :$v_3$] at (D) {};
		\node[label = below :$G:$] at (Z) {};
\end{tikzpicture}}$$
  The matrix $\Phi(G)$ is
$$\Phi(G)= \begin{pmatrix} 1&0&0\\ 0&1&1 \\ 0&0&1 \end{pmatrix}$$ The
  figure below shows the graph $\varphi(G^3)$.

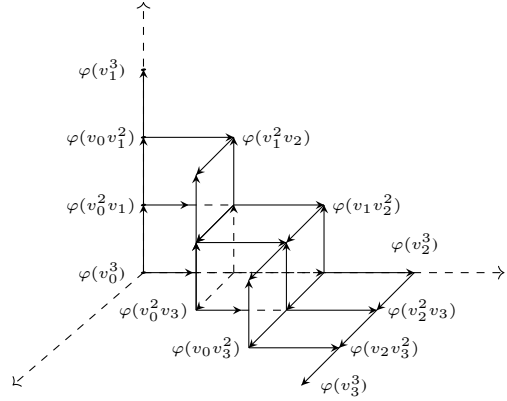
\begin{center}
	\begin{tikzpicture}[x={(0.6cm,0cm)}, y={(0.25cm,0.25cm)}, z={(0cm,0.45cm)}, decoration={markings, 
			mark= at position 1 with {\arrow{stealth}},
		}
		] 	
		\coordinate (dx) at (1,0,0);
		\coordinate (dy) at (0,1,0);
		\coordinate (dz) at (0,0,1);
		
		\coordinate (P111) at (0,0,0);
		\coordinate (P112) at (0, 0,2);
		\coordinate (P113) at (2, 0, 0); 
		\coordinate (P114) at (2, -2, 0);
		\coordinate (P122) at  (0, 0,4);
		
		\coordinate (P123) at (2, 0, 2);
		\coordinate (P124) at (2, -2, 2);
		\coordinate (P222) at (0, 0, 6);
		\coordinate (P223) at (2, 0, 4);
		\coordinate (P224) at (2,- 2,4);
		
		\coordinate (P133) at (4, 0, 0);
		\coordinate (P134) at (4, -2, 0);
		\coordinate (P144) at (4, -4,0);
		
		\coordinate (P233) at (4, 0,2);
		\coordinate (P234) at (4, -2,2);
		\coordinate (P244) at (4, -4,2);
		\coordinate (P333) at (6, 0,0);
		\coordinate (P334) at (6, -2,0 );
		\coordinate (P344) at (6, -4,0 );
		\coordinate (P444) at (6, -6,0 );
		
		\node at (P111)  [label= left:\tiny{$\varphi(v_0^3$)}] {};
		\node[label = {[label distance=-4pt]  left :\tiny $\varphi(v_0^2v_1)$}] at (P112) {};
		\node at (P113) [label=left:\tiny{}] {};
		\node[label = {[label distance=-4pt]  left :\tiny $\varphi(v_0^2v_3)$}] at (P114) {};
		\node[label = {[label distance=-4pt]  left :\tiny $\varphi(v_0v_1^2)$}] at (P122) {};
		\node at (P123) [label= \tiny{}] {};
		\node at (P124) [label= \tiny{}] {};
		\node at (P222) [label=left:\tiny{$\varphi(v_1^3$)}] {};
		\node[label = {[label distance=-4pt]  right :\tiny $\varphi(v_1^2v_2)$}] at (P223) {};
		\node at (P224) [label=right:\tiny{}] {};
		\node at (P133) [label=\tiny{}] {};
		\node at (P134) [label=  right:\tiny{}] {};
		\node[label = {[label distance=-4pt]  left :\tiny $\varphi(v_0v_3^2)$}] at (P144) {};
		\node[label = {[label distance=-4pt]  right :\tiny $\varphi(v_1v_2^2)$}] at (P233) {};
		\node at (P234) [label=right: \tiny{}] {};
		\node at (P244) [label=\tiny{}] {};
		\node at (P333) [label= \tiny{$\varphi(v_2^3$)}] {};
		\node[label = {[label distance=-3pt]  right :\tiny $\varphi(v_2^2v_3)$}] at (P334) {};
		\node[label = {[label distance=-3pt]  right :\tiny $\varphi(v_2v_3^2)$}] at (P344) {};
		\node at (P444) [label=right:\tiny{$\varphi(v_3^3$)}] {};
		
		\draw[black, fill=black] (P111) circle(0.04);	 	
		\draw[black, fill=black] (P112) circle(0.04);
		\draw[black, fill=black] (P122) circle(0.04);
		\draw[black, fill=black] (P222) circle(0.04);
		
		\draw[postaction={decorate}](P111) -- (P112);
		\draw[postaction={decorate}](P112) -- (P122);
		\draw[postaction={decorate}](P122) -- (P222);
		\draw[postaction={decorate}](P122) -- (P223);

		\draw[postaction={decorate}](P123) -- (P223);
		\draw[postaction={decorate}](P114) -- (P124);
		\draw[postaction={decorate}] (P124) -- (P224);
		\draw[postaction={decorate}](P223) -- (P224);
		\draw[postaction={decorate}](P123) -- (P124);

		\draw[postaction={decorate}](P123) -- (P233);
		\draw[postaction={decorate}](P233) -- (P234);
		\draw[postaction={decorate}] (P124) -- (P234);
		\draw[postaction={decorate}](P123) -- (P124);
		
		\draw[postaction={decorate}] (P114) -- (P124);
		\draw[postaction={decorate}](P133) -- (P233);
		\draw[postaction={decorate}] (P134) -- (P234);
		\draw[dashed][postaction={decorate}](P113) -- (P123);

		\draw [dashed][postaction={decorate}] (P113) -- (P133);
		\draw [dashed][postaction={decorate}] (P113) -- (P114);
		
		\draw[postaction={decorate}](P133) -- (P134);

		\draw[postaction={decorate}](P133) -- (P333);
		\draw[postaction={decorate}](P134) -- (P334);
		\draw[postaction={decorate}](P333) -- (P334);
		\draw[postaction={decorate}] (P334) -- (P344);
		\draw[postaction={decorate}](P144) -- (P344);
		\draw[postaction={decorate}] (P134) -- (P144);
		\draw[postaction={decorate}](P144) -- (P244);
		\draw[postaction={decorate}] (P234) -- (P244);
		\draw[postaction={decorate}](P344) -- (P444);

		\coordinate (R) at (1.1, 0,0 );
		\coordinate (S) at (1, 0,2 );
		\coordinate (T) at (3, -2,0 );	
		
		\node at (R) [label=right:\tiny{}] {}; 
		\node at (R) [label=right:\tiny{}] {};
		\node at (R) [label=right:\tiny{}] {};
		
		\draw[postaction={decorate}] (P111) -- (R);
		\draw[dashed] (R) -- (P113);
		\draw[postaction={decorate}](P112) -- (S);
		\draw[dashed] (S) -- (P123);
		\draw[postaction={decorate}] (P114) -- (T);
		\draw[dashed] (T) -- (P134);

		\coordinate (U) at (-0.4,- 6, 0 );
		\coordinate (V) at (8,0 ,0 );
		\coordinate (W) at (0, 0,8 );	
		
		\node at (U) [label=right:\tiny{}] {}; 
		\node at (V) [label=right:\tiny{}] {};
		\node at (W) [label=right:\tiny{}] {};
		
		\draw[->, dashed] (P111) -- (U);
		\draw[->, dashed] (P111) -- (V);
		\draw[->, dashed] (P111) -- (W);
	
	\end{tikzpicture}
 \end{center}
\end{example}

\begin{bfchunk}{Building the cubical polyhedral cell complex
  $\ol{G^r}$ in  $\RR^q$.}
  \label{ss:graphpowers}  Now that we have a directed graph $G^r$, we focus on
     building an acyclic polyhedral cell complex which has $G^r$
     as its $1$-skeleton. 
     Using the concept of a sink (see \cref{d:cube}), we identify the $1$-skeleta of cubes
   that appear as induced subgraphs of $\varphi(G^r)$. Below, by the
   {\bf support} of $\bva$ or of $\ba=(a_0,a_1,\ldots,a_q) \in
   \ZZ_{\geq 0}^{q+1}$ we mean the set $$\supp(\bva)=\supp(\ba)=\{j>0
   \mid a_j \neq 0\} \subseteq [q].$$
   \end{bfchunk}

\begin{definition}[{\bf The subgraph $\cC(\bb,B)$ of $\varphi(G^r)$ }]\label{d:cubes}
  If $G$ is a directed rooted tree labeled as in \cref{n:direction},
  $r,q>0$, $\bb \in \cNr$ and $\varnothing \neq B \subseteq
  \supp(\bb)$, then  we denote by $\cC(\bb,B)$ the induced directed subgraph  of
  $\varphi(G^r)$ on the vertex set
     $$\{\varphi(\bb)-\sum_{i\in B'} \be_{i} \mid \varnothing \subseteq B' \subseteq
  B\}.$$
 \end{definition}

\begin{proposition}\label{p:cubes}
Let $G$ be a directed rooted tree as in \cref{n:direction},
  $r,q>0$, $\bb \in \cNr$ and $\varnothing \neq B \subseteq
  \supp(\bb)$.
    Then $\cC(\bb, B)$ is the $1$-skeleton of a $|B|$-dimensional
     cube in $\RR^q$ with source $\varphi(\ba)$, where
      $$ \ba=\bb - \sum_{i \in B}(\bff_i - \bff_{\tau(i)}), $$ sink
     $\varphi(\bb)$, and edges
     $$\big [\varphi(\ba)+\sum_{i\in B'} \be_{i} , \ \varphi(\ba)+\sum_{i\in
     	B'} \be_{i} + \be_k \big ] \qfor \varnothing \subseteq B' \subsetneq B \qand k \in B \ssm
     B'.$$  
 \end{proposition}

\begin{proof}
   
     Let $\bb=(b_0,b_1,\ldots,b_q)$ and set
     \begin{equation}\label{equ:ab}
   \ba=\bb+\sum_{i \in B} \bff_{\tau({i})}-\sum_{i \in
     B}\bff_{i} \end{equation} as in the statement of the
     theorem. Since $B \subseteq \supp(\bb)$,
     $\ba=(a_0,a_1,\ldots,a_q)$ in $\cNr$ and $\varphi(\ba) \in
     \varphi(G^r)$.  It follows that $a_{\tau(i)}>0$ for every $i \in
     B$, and by \cref{l:j-unique} and \cref{l:differbyj},
     $\varphi(\ba) = \varphi(\bb) - \sum_{i \in B} \be_i$. Therefore
  $$V=\{\varphi(\ba)+\sum_{i\in B'} \be_{i} \mid \varnothing \subseteq
     B' \subseteq B\}.$$

     Moreover, for each $B' \subseteq B$,
 $$\varphi(\ba) + \sum_{i \in B'} \be_i = \varphi \big ( \ba +\sum_{i
   \in B'}(\bff_i-\bff_{\tau({i})})\big )\in \varphi(G^r).$$ By
 \eqref{e:edges}, $[\bv^\bc, \bv^\bd]$ is an edge of $G^r$ if and only if $\bd
 = \bc + \bff_k - \bff_{\tau(k)}$ for some $k$, or equivalently
 $\varphi(\bd) = \varphi(\bc) + \be_k$. Thus the edges of the induced
 graph on vertex set $V$ are precisely
   $$\big [\varphi(\ba)+\sum_{i\in B'} \be_{i} ,
   \ \varphi(\ba)+\sum_{i\in B'} \be_{i} + \be_k \big ] \qfor
 \varnothing \subseteq B' \subsetneq B \qand k \in B \ssm B'.$$ By
 \cref{d:cube}, these are the edges of a  $|B|$-dimensional cube
   in $\RR^q$ with source $\varphi(\ba)$ and sink
   $\varphi(\bb)=\varphi(\ba + \sum_{i \in B} \be_i)$, as desired.
  \end{proof}

Consider the graph $\varphi(G^r)$ as a union of $1$-skeleta
of cubes embedded in $\RR^q$, as described in \cref{p:cubes}.
Our next
step is to fill each of these cubes to build a polyhedral cell
complex. With this goal in mind, we set the following notation.

\begin{notation} \label{filledC} Let $\ocC(\bb,B)$  with $B \subseteq \supp(\bb)$ denote  the solid cube whose $1$- skeleton is $\cC(\bb,B)$ from Proposition \ref{p:cubes}.
 Let $\ol{G^r}$ be the collection of the
solid cubes $\ol{\cC}(\bb,B)$  for each $\bb \in \cNr$ and  $B \subseteq \supp(\bb)$. More precisely
\begin{equation}\label{e:cell}
  \ol{G^r}=\bigcup_{\tiny \begin{array}{cc}\bb \in \cNr\\ B \subseteq \supp(\bb) \end{array}}\{ \ol{\cC}(\bb,B)\}. 
\end{equation}
Note that if $B=\varnothing$, then we set $\ol{\cC}(\bb,B)=\cC(\bb,B)$ to be
the $0$-dimensional cube $\{\varphi(\bb)\}$.
\end{notation}

\begin{example}\label{ex:path-2}
If $\varphi(G^r)$ is as in \cref{ex:path-0}, then $\ol{G^2}$ is shown below.  The $2$-cell is $\ocC((0,1,1),\{1,2\})$, with $$
\text{source} \quad \varphi(1,1,0)=\varphi(v_0v_1)=
\begin{pmatrix}
  1\\0
\end{pmatrix}
\quad \text{and sink} \quad
 \varphi(0,1,1)=\varphi(v_1v_2)=
\begin{pmatrix}
  2\\1
\end{pmatrix}.$$ Taking ${\bf a}=(1,1,0)$, the edges of the $2$-cube $\ocC((0,1,1),\{1,2\})$ are $$\big [\varphi(\ba) , \ \varphi(\ba)+ \be_{1}], \big [\varphi(\ba) , \ \varphi(\ba)+ \be_{2}], \big [\varphi(\ba) + \be_{1} , \ \varphi(\ba)+ \be_{1}+ \be_{2}], \big [\varphi(\ba) + \be_{2} , \ \varphi(\ba)+ \be_{2}+ \be_{1}],$$
or equivalently, the edges can be written as
$$[(1,0), (2,0)], [(1,0), (1,1)], [(2,0), (2,1)], [(1,1), (2,1)].$$
Similarly, the 1-cell $\ocC((0,1,1),\{1\})$ is the top edge of the square, with source $\varphi(1,0,1) = \varphi(v_0v_2) = (1,1)^T$.
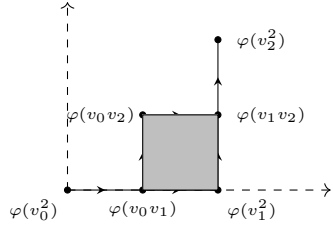
\begin{center}
	\begin{tikzpicture} [decoration={markings, 
			mark= at position 0.5 with {\arrow{stealth}},
		}
		] 	
		\coordinate (A') at (3, 0);
		\coordinate (B') at (4, 1);
		\coordinate (C') at (5, 1);
		\coordinate (D') at (4, 0);
		\coordinate (E') at (5, 0);
		\coordinate (F') at (5, 2);
		\coordinate (X') at (6.5,0);
		\coordinate (Y') at (3,2.5);
		\draw[black, fill=black] (A') circle(0.04);
		\draw[black, fill=black] (B') circle(0.04);
		\draw[black, fill=black] (C') circle(0.04);
		\draw[black, fill=black] (D') circle(0.04);
		\draw[black, fill=black] (E') circle(0.04);
		\draw[black, fill=black] (F') circle(0.04);
		\draw[postaction={decorate}](D') -- (B');
		\draw[postaction={decorate}] (B') -- (C');
		\draw[postaction={decorate}](E') -- (C');
		\draw[postaction={decorate}] (A') -- (D');
		\draw[postaction={decorate}](D') -- (E');
		\draw[postaction={decorate}] (C') -- (F');
		\draw[->, dashed] (A') -- (X');
		\draw[->, dashed] (A') -- (Y');
		\node[label = {[label distance=-5pt] below left :\tiny $\varphi(v_0^2)$}] at (A') {};
		\node[label = {[label distance=-4pt]  left :\tiny $\varphi(v_0v_2)$}] at (B') {};
		\node[label = right :\tiny{$\varphi(v_1v_2)$}]at (C') {};
		\node[label = {[label distance=-3pt]  below :\tiny $\varphi(v_0v_1)$}] at (D') {};
		\node[label = {[label distance=-5pt]  below right :\tiny $\varphi(v_1^2)$}] at (E') {};
		\node [label = right :\tiny{$\varphi(v_2^2)$}] at (F') {};
		\draw[fill=lightgray] (B') -- (C') -- (E')  -- (D') -- cycle;
	\end{tikzpicture}
\end{center}
\end{example}

\begin{bfchunk}{Constructing $\ol{G^{r+1}}$ from $\ol{G^r}$ when $r\ge q$.}\label{ss:covering} Note that, for any $\bb\in \cN_r$ and $B\subseteq \supp(\bb)$,
  $$\ocC(\bb, B))=\ocC(\bb+\bff_0, B)\,, {\text{ where }} \bff_0 = (1,
  0, \dots, 0),$$ hence there is an embedding $\ol{G^r}\subseteq
  \ol{G^{r+1}}$.  Since $\bb \in \cN_r$, we always have
    $|\supp(\bb)| \leq r$. On the other hand when $r\le q$, there
  exists at least one $\bb\in \cN_r$ such that $|\supp(\bb)|=r$, hence
  $\ocC(\bb,\supp(\bb))$ is a maximal, $r$-dimensional cube in
  $\ol{G^r}$. Thus, each iteration of the construction of $\ol{G^r}$
  adds new, higher dimensional cubes, as long as $r\le q$.  We refer
  the reader to the picture of $\varphi(G^3)$ in \cref{e:G^3} in order
  to visualize the $3$-dimensional cube in $\ol{G^3}$ for that
  example. On the other hand, the dimension of the cubes cannot exceed
  the dimension $q$ of the space, so the maximal dimension of a cube
  stabilizes at $q$. This does not mean that $\ol{G^r}$ and
  $\ol{G^{r+1}}$ become equal, but it turns out that, when $r \geq q$,
  all the cubes of $\ol{G^{r+1}}$ can be described as translations of
  cubes in $\ol{G^{r}}$, as discussed below. 

We denote by  $\ocC(\bb, B)+{\bx}$ the translation by the vector $\bx\in \mathbb R^q$ of a cell $\ocC(\bb, B)\in \ol{G^r}$.  For any $0\le j\le q$, $\bb\in \cNr$, and $B\subseteq \supp(\bb)$, observe that 
\begin{equation}
\label{e:translate-c}
\ocC(\bb, B)+ \varphi(\bff_j)=\ocC(\bb+\bff_j, B).
\end{equation}
 In particular, translation by the vector $\varphi(\bff_j)$ in $\mathbb R^q$ induces a cellular map $t_j\colon \ol{G^r}\to \ol{G^{r+1}}$, satisfying 
$$
t_j(\ocC(\bb, B)) =\ocC(\bb+\bff_j, B)\in \ol{G^{r+1}}.
$$
When $r\ge q$, notice that all the cubes in $\ol{G^{r+1}}$ are obtained as translations of cubes in $\ol{G^{r}}$, namely: 
\begin{equation}
\label{G^r-cover}
\bigcup_{0\le j\le q} t_j(\ol{G^r})=\ol{G^{r+1}} \qforall r\ge q\,.
\end{equation}
Indeed, let $\ba\in \cN_{r+1}$ and $A\subseteq \supp(\ba)$. If $r\ge
q$, then, since $a_0+\dots a_q=r+1$, either $a_i>1$ for some $i\in
[q]$ or else $a_0\ge 1$. When $a_0\ge 1$ and $a_j\le 1$ for all $j\in
[q]$, set $i=0$. With $i$ as above, one has $\ba -\bff_i\in \cN_r$,
$A\subseteq \supp(\ba)=\supp(\ba-\bff_i)$, and
$$
\ocC(\ba, A)=t_i(\ocC(\ba-\bff_i, A))\in t_i(\ol{G^r}),
$$
which proves \cref{G^r-cover}. 
\end{bfchunk}

\begin{example}
In the context of Example \ref{ex:path-2}, 
$$
\varphi(\bff_0)=\begin{pmatrix} 0\\ 0\end{pmatrix}, \quad \varphi(\bff_1)=\begin{pmatrix} 1\\ 0\end{pmatrix},  \quad \varphi(\bff_2)=\begin{pmatrix} 1\\ 1\end{pmatrix}
$$
The cell complex $\ol{G^3}$ is shown below, and consists of three overlapping copies of $\ol{G^2}$, obtained by translating $\ol{G^2}$ along the vectors listed above.   Note that $t_0(\ol{G^2})=\ol{G^2}$ and the centermost vertex in the picture is $\varphi(v_0v_1v_2)$.

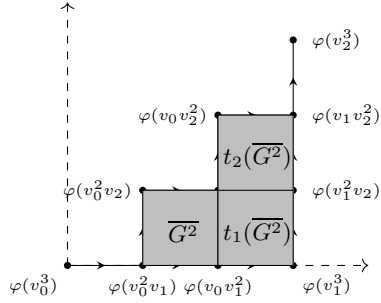
\begin{center}
\begin{tikzpicture} [decoration={markings, 
	mark= at position 0.5 with {\arrow{stealth}},
	}
] 	
\coordinate (A') at (3, 0);
\coordinate (B') at (4, 1);
\coordinate (C') at (5, 1);
\coordinate (D') at (4, 0);
\coordinate (E') at (5, 0);
\coordinate (F') at (5, 2);
\coordinate (X') at (7,0);
\coordinate (Y') at (3,3.5);
\coordinate (M') at (6, 0);
\coordinate (N') at (6, 1);
\coordinate (P') at (6, 2);
\coordinate (Q') at (6, 3);
\draw[black, fill=black] (A') circle(0.04);
 \draw[black, fill=black] (B') circle(0.04);
 \draw[black, fill=black] (C') circle(0.04);
 \draw[black, fill=black] (D') circle(0.04);
 \draw[black, fill=black] (E') circle(0.04);
 \draw[black, fill=black] (F') circle(0.04);
 \draw[black, fill=black] (M') circle(0.04);
 \draw[black, fill=black] (N') circle(0.04);
 \draw[black, fill=black] (P') circle(0.04);
 \draw[black, fill=black] (Q') circle(0.04);
\draw[postaction={decorate}](D') -- (B');
\draw[postaction={decorate}] (B') -- (C');
\draw[postaction={decorate}](E') -- (C');
\draw[postaction={decorate}] (A') -- (D');
\draw[postaction={decorate}](D') -- (E');
\draw[postaction={decorate}] (C') -- (F');
\draw[postaction={decorate}] (M') -- (N');
\draw[postaction={decorate}] (N') -- (P');
\draw[postaction={decorate}] (P') -- (Q');
\draw[postaction={decorate}] (E') -- (M');
\draw[postaction={decorate}] (C') -- (N');
\draw[postaction={decorate}] (F') -- (P');
\draw[->, dashed] (A') -- (X');
\draw[->, dashed] (A') -- (Y');
\node[label = {[label distance=-5pt] below left :\tiny $\varphi(v_0^3)$}] at (A') {};
\node[label = {[label distance=-3pt] left :\tiny $\varphi(v_0^2v_2)$}] at (B') {};
\node[label = {[label distance=-3pt]  below :\tiny $\varphi(v_0^2v_1)$}] at (D') {};
\node[label = {[label distance=-3pt]  below :\tiny $\varphi(v_0v_1^2)$}] at (E') {};
\node[label = {[label distance=-5pt]  below right :\tiny $\varphi(v_1^3)$}] at (M') {};
\node[label = {[label distance=-3pt]  left :\tiny $\varphi(v_0v_2^2)$}] at (F') {};
\node [label = right :\tiny{$\varphi(v_2^3)$}] at (Q') {};
\node[label = right :\tiny{$\varphi(v_1^2v_2)$}] at (N') {};
\node[label = right :\tiny{$\varphi(v_1v_2^2)$}] at (P') {};
\draw[fill=lightgray] (B') -- (C') -- (E')  -- (D') -- cycle;
\draw[fill=lightgray] (C') -- (N') -- (M')  -- (E') -- cycle;
\draw[fill=lightgray] (F') -- (P') -- (N')  -- (C') -- cycle;
\pgfputat{\pgfxy(4.34,.45)}{\pgfbox[left,center]{\small $\ol{G^2}$}}
\pgfputat{\pgfxy(5.07,.45)}{\pgfbox[left,center]{\small $t_1(\ol{G^2})$}}
\pgfputat{\pgfxy(5.07, 1.45)}{\pgfbox[left,center]{\small $t_2(\ol{G^2})$}}
\end{tikzpicture}
\end{center}
\end{example}

  \begin{proposition} [{\bf The faces of the cube $\ocC(\bb,B)$}]\label{p:faces}
  Let $r,q>0$, and let $G$ be a directed tree on $q+1$ vertices
  labeled as in \cref{n:direction}, $\bb \in \cNr$, and $\varnothing
  \neq B \subseteq \supp(\bb)$. Then for every face $F$ of
  $\ocC(\bb,B)$ there exists a subset $C \subseteq B$ such
  that
$$F=\ocC(\bc, B \setminus C) \qwhere \bc = \bb - \sum_{i \in C'} \big
  (\bff_i - \bff_{\tau({i})}\big ) \qforsome \varnothing \subseteq C'
  \subseteq C.$$ In particular, for a $(|B|-1)$-dimensional face $F$
  of $\ocC(\bb,B)$ there is some $i \in B$ such that
     $$F=\ocC(\bb,B\ssm\{i\}) \qor
  F=\ocC(\bb-\bff_i+\bff_{\tau(i)},B\ssm\{i\}).$$
  \end{proposition}

 \begin{proof} 
  Let $\varphi(\ba)$ be the source of $\cC(\bb,B)$ as in
  \cref{p:cubes}.  Assume that $F$ has dimension $|B|-t$ for $t \geq
  1$. We use induction on $t$ to prove the statement. The base case is
  $t=1$. By \cref{d:cube} for a fixed $i \in B$, the face $F$ can be
  described as the cube on vertices
  $$\big \{\varphi(\ba) + \be_i + \sum_{j\in B' } \be_{j} \mid B'
  \subseteq B\ssm \{i\} \big \} \qor \big \{ \varphi(\ba) + \sum_{j\in
    B' } \be_{j} \mid B' \subseteq B\ssm \{i\} \big \}.$$ Using
  \eqref{e:edges}, it follows immediately that
  $$F=\ocC(\bb,B\ssm\{i\}) \qor
  F=\ocC(\bb-\bff_i+\bff_{\tau(i)},B\ssm\{i\}).$$ If we set $C=\{i\}$,
  then indeed we have shown that $$F=\ocC(\bc, B \setminus C) \qwhere
  \bc = \bb - \sum_{i \in C'} \big (\bff_i - \bff_{\tau({i})}\big )
  \qand C'= \varnothing \qor C'=\{i\}.$$

  Now suppose $t>1$, and $F$ is of dimension $|B|-t$. By the definition of a cube (\cref{d:cube}), $F$ is a the boundary of a $|B|-t+1=|B|-(t-1)$-dimensional cube $G$, which is itself a face of $\ocC(\bb,B)$. 
  By the
  induction hypothesis,  there exists a subset $D \subseteq B$ such
  that $$G=\ocC(\bd, B \ssm D) \qwhere \bd = \bb - \sum_{i \in D'}
  \big (\bff_i - \bff_{\tau({i})}\big ) \qforsome
  \varnothing \subseteq D' \subseteq D.$$
  By the base case of the induction, we 
  know  for some $i \in B\ssm D$,
     $$F=\ocC(\bd,B\ssm(D\cup \{i\})) \qor
  F=\ocC(\bd-\bff_i+\bff_{\tau(i)},B\ssm (D\cup \{i\})).$$ Setting
  $C=D \cup \{i\}$, we have the desired result where $C'=D'$ or $C' = D'\cup \{i\}$.
   \end{proof}

\begin{proposition}[{\bf $\ol{G^r}$ is a polyhedral cell complex}]
\label{p:poly-cell-cx}
    Let $r,q>0$ and let $G$ be a tree on $q+1$ vertices  as in
      \cref{n:direction}. Then $\ol{G^r}$ is a polyhedral cell
    complex.
\end{proposition}

\begin{proof}
  \cref{d:pcx}~(1) follows directly from \cref{p:faces}.  To verify
  \cref{d:pcx}~(2),  for $\bb, \bd \in \cNr$
    consider $$F=\ocC(\bb,B)\cap \ocC(\bd,D) \qwhere \varnothing \neq
    B \subseteq \supp(\bb) \qand \varnothing \neq D \subseteq
    \supp(\bd).$$ We need to show $$F=\ocC(\bc,C)  \qforsome \bc \in
    \cNr \qand C \subseteq \supp(\bc).$$  By \cref{p:cubes}, the cubes 
    $\ocC(\bb,B)$ and $\ocC(\bd,D)$ have vertex sets $$\{\varphi(\bb)-
    \sum_{j\in A} \be_{j} \mid A \subseteq B \} \qand \{\varphi(\bd)-
    \sum_{j\in A} \be_{j} \mid A \subseteq D\},$$ respectively.
  
  {\bf Claim:} if $B_1, B_2 \subseteq B$ are such that
  \begin{equation}\label{e:1}
    u=\varphi(\bb)- \sum_{j\in B_1} \be_{j} \in F \qand v=\varphi(\bb)-
    \sum_{j\in B_2} \be_{j} \in F,
  \end{equation} then
  \begin{equation}\label{e:intersect}
    \varphi(\bb)- \sum_{j\in B_1 \cap B_2} \be_{j} \in F.
  \end{equation}

  To see this,  first note that $$\varphi(\bb)- \sum_{j\in B_1 \cap B_2} \be_{j} \in  \ocC(\bb, B)$$ by definition. Now since $u,v \in F$, then  $u, v \in
  \ocC(\bd,D)$. Hence, for some $D_1, D_2 \subseteq D$
  \begin{equation}\label{e:2}
  u=\varphi(\bd)- \sum_{j\in D_1} \be_{j}
  \qand v=\varphi(\bd)- \sum_{j\in D_2} \be_{j}.
  \end{equation}
  Now \eqref{e:1} tells us that $$u- \sum_{j\in B_2\ssm B_1} \be_{j}=
  v - \sum_{j\in B_1\ssm B_2} \be_{j}$$ which combined with
  \eqref{e:2} implies that
  $$\varphi(\bd)- \sum_{j\in D_1} \be_{j} - \sum_{j\in B_2\ssm B_1}
  \be_{j}= \varphi(\bd)- \sum_{j\in D_2} \be_{j} - \sum_{j\in
    B_1\ssm B_2} \be_{j}.$$ This implies that $$D_1 \cup (B_2\ssm
  B_1) = D_2 \cup (B_1\ssm B_2) \Longrightarrow (B_1\ssm B_2) \subseteq D_1.$$
  Now combining \eqref{e:1} and \eqref{e:2} we get 
  $$\varphi(\bb)- \sum_{j\in B_1\ssm B_2} \be_{j} - \sum_{j\in
    B_1\cap B_2} \be_{j} = u = \varphi(\bd)- \sum_{j\in D_1} \be_{j}.$$
  Moreover, since $(B_1\ssm B_2) \subseteq D_1$, it follows that
  $$ \varphi( \bb)- \sum_{j\in B_1\cap B_2} \be_{j}= \varphi(\bd) -
  \sum_{j\in D_1\ssm (B_1\ssm B_2)} \be_{j} \in \ocC(\bb, D) \cap
    \ocC(\bd, D).$$ This establishes \eqref{e:intersect} as
  claimed.  It follows immediately that there exist unique minimal
  subsets $B_0$ of $B$ and $D_0$ of $D$ such that
  $$ \varphi({\bb}) - \sum_{j \in B_0} \be_j =\varphi({\bd}) - \sum_{j \in
    D_0} \be_j \in F.$$ 
Set
  $$\bc=\bb - \sum_{j \in B_0} (\bff_j-\bff_{\tau(j)})\in \cNr
  \qand C= (B\ssm B_0)\cap (D\ssm D_0) \subseteq \supp(\bc)$$
  so that $$\varphi (\bc) =\varphi(\bb) - \sum_{j \in B_0} \be_j \in F,$$
  and moreover
   $$\varphi(\bc) -\sum _{j\in C'} \be_j \in F
   \qforall C' \subseteq  C.$$  In other words,
   \begin{equation}\label{e:one-direction}\ocC(\bc, C) \subseteq F.
   \end{equation}

   To see the reverse inclusion to \eqref{e:one-direction}, take  $w
 \in F$ so that
 $$w=\varphi(\bb) - \sum_{j \in B'} \be_j =\varphi(\bd) - \sum_{j \in
   D'} \be_j $$ for some $$B_0 \subseteq B' \subseteq B \qand D_0
 \subseteq D' \subseteq D.$$ So we can write
 \begin{equation}\label{e:w1} w=\varphi(\bb) -
 \sum_{j \in B_0} \be_j - \sum_{j \in B'\ssm B_0} \be_j
 =\varphi(\bc)-\sum_{j \in B'\ssm B_0} \be_j. 
 \end{equation}
 Similarly
 \begin{equation}\label{e:w2} w=\varphi(\bc)-\sum_{j \in D'\ssm D_0} \be_j.
 \end{equation}
 \cref{e:w1,e:w2} indicate that $$\sum_{j \in B'\ssm B_0} \be_j=\sum_{j \in D'\ssm D_0} \be_j \Longrightarrow B'\ssm B_0=D'\ssm D_0 \subseteq C.$$ Hence $w\in \ocC(\bc,C)$. We conclude that $F=\ocC(\bc, C)$ as desired.
  \end{proof}

\begin{bfchunk}{Labeling $\ol{G^r}$ with monomials.}\label{ss:homogenize}
Recall that if $I$ is a monomial ideal generated by $m_0,\ldots,m_q$,
the ideal $I^r$ is generated by monomials of the
form $$\bma=m_0^{a_0}\cdots m_q^{a_q}$$ where $\bm=\{m_0,\ldots,m_q\}$
and $\ba=(a_0,a_1,\ldots,a_q) \in \cNr$.  When $I$ is of projective
dimension one, \cite[Proposition 4.1]{CEFMMSS} shows that $$\bma=\bmb
\iff \ba=\bb.$$ In particular, the set $\{\bma \mid \ba \in \cNr\}$ is
a minimal generating set for the ideal $I^r$.

Suppose $G$ is the directed tree with root $v_0$ and vertex set
$v_0,v_1,\ldots,v_q$ that supports a minimal free resolution of the
square-free monomial ideal $I=(m_0,m_1,\ldots,m_q)$ of projective
dimension one, where every vertex $v_i$ is labeled with the monomial
$m_i$. Our goal is to show that the polyhedral cell complex
$\ol{G^r}$, where every vertex $\bva$ is labeled with the monomial
$\bma$, supports a minimal free resolution of $I^r$.

With this monomial labeling, we homogenize the cellular (cubical)
chain complex of $\ol{G^r}$.  The cell complex $\ol{G^r}$ gives rise
to the oriented chain complex $C(\ol{G^r},\sfk)$, as
described in \cref{oriented-chain-complex}. In order to define the
signs of the maps, let $$B=\{j_1, \ldots, j_i\} \qwhere j_1 < j_2 <
\cdots < j_i.$$ Using \cref{p:faces} the differential $\partial_i$ is
described by:
\begin{equation}
\label{chain}
\partial_i(\bu_{\cC(\bb,B)})=\sum_{1 \leq k \leq i}(-1)^{k+1}
\bu_{\cC(\bb,B\ssm \{j_k\})}+\sum_{1 \leq k \leq i}(-1)^{k}\bu_{\cC(\bb-\bff_{j_k}+\bff_{\tau(j_k)},B\ssm\{j_k\})}.
\end{equation}

To homogenize these maps, denote the lcm of the monomial labels of
the vertices of a cube $\ocC(\bb,B)$ by $\bm_{\cC(\bb,B)}$.
The chain complex in \eqref{chain} then homogenizes as described in
Section~\cref{s:homogenization} to a $\mathbb Z^n$-graded cellular complex
\begin{equation}\label{e:cell-complex} \mathbb F_{\ol{G^r}}: \qquad \cdots \rightarrow
F_i\xrightarrow{\partial_i}F_{i-1}\to\dots\to
F_1\xrightarrow{\partial_1}F_0
\end{equation} where $F_i$ is the free graded $R$-module with basis elements
$\bu_{\cC(\bb,B)}$ with $\ocC(\bb,B)\in {\ol{G^r}}$ and where 
$\bu_{\cC(\bb,B)}$ is considered to be in degree equal to the exponent
vector of $\bm_{\cC(\bb,B)}$. For each $i>0$ the differential
$\partial_i$ of $\mathbb F_{\ol{G^r}}$ is described by:

\begin{align}
\begin{split}
\label{res}
\partial_i(\bu_{\cC(\bb,B)})=
&\sum_{1 \leq k \leq i}(-1)^{k+1} \frac{\bm_{\cC(\bb,B)}}{\bm_{\cC(\bb, B\ssm \{j_k\})}}
\ \bu_{\cC(\bb,B\ssm \{j_k\})}+\\
+&\sum_{1 \leq k \leq i}(-1)^{k}\frac{\bm_{\cC(\bb,B)}}{\bm_{\cC(\bb-\bff_{j_k}+\bff_{\tau(j_k)},B\ssm\{j_k\})}}
\ \bu_{\cC(\bb-\bff_{j_k}+\bff_{\tau(j_k)},B\ssm\{j_k\})}\,.
\end{split}
\end{align}
\end{bfchunk}
 
	We now focus on the two monomial coefficients appearing in \eqref{res}.
	
	\begin{lemma}
		\label{p:simplify}
		Let $\bb\in  \cNr$, $B\subseteq \supp(\bb)$ and $i\in B$.
		The following equalities then hold: 
		\begin{enumerate}[\quad\rm(1)]
			\item ${\displaystyle \frac{\bm_{\cC(\bb, B)}}{\bm_{\cC(\bb, B\ssm\{i\})}}
				=\frac{\lcm(m_i, m_{\tau(i)})}{m_i}}$;

			\item ${\displaystyle
				\frac{\bm_{\cC(\bb,B)}}{\bm_{\cC(\bb-\bff_i+\bff_{\tau(i)}, B\ssm\{i\})}}
				=\frac{\lcm(m_i, m_{\tau(i)})}{m_{\tau(i)}}}$.
		\end{enumerate}
	\end{lemma}

\begin{proof}  By \cref{p:cubes} and \eqref{e:edges} the vertices of
		$\cC(\bb,B)$ are the images under $\varphi$ of 
		
		$$ \bb- \sum_{j\in B'} \bff_{j} +\sum_{j\in B'}
  \bff_{\tau(j)} \quad {\text{ for all }} B' {\text{ such that }}
  \varnothing \subseteq B' \subseteq B.$$ Notice that for each $j$,
  the monomial label associated to  $$\bb - \bff_j +
    \bff_{\tau(j)} \quad \text{ is } \quad \frac{\bmb
      m_{\tau(j)}}{m_j},$$ so for each $B' \subseteq B$, the label
  associated to
			$$ \bb- \sum_{j\in B'} \bff_{j} +\sum_{j\in B'}
		\bff_{\tau(j)} \quad {\text{ is }} \quad \bmb \prod_{j \in B'}\frac{m_{\tau(j)}}{m_{j}}.$$
		As a result, considering all $2^{|B|}$ vertices of $\cC(\bb,B)$, we have
		\begin{align} \nonumber
			\bm_{\cC(\bb,B)}&=\lcm \left \{
			\bmb \prod_{j \in B'}\frac{m_{\tau(j)}}{m_{j}} \ \big | \  \varnothing \subseteq B' \subseteq B
			\right \}.\\ \nonumber
	 &=\lcm \left (
			\bmb,
			\left \{\frac{\bmb \cdot m_{\tau(j_1)}\cdots m_{\tau(j_t)}}{m_{j_1} \cdots m_{j_t}}\right \}_{ \{j_1,\ldots,j_t \}\subseteq B}
			\right )\\ &= \bmb \prod_{j\in B} \left( \prod_{x
				\mid m_{\tau(j)}, \, x \nmid m_j} x \right ) \qwhere x \in \{x_1, \dots, x_n\}.
			\label{e:monomial-label}
		\end{align}
		
		Similarly,
		$$ \bm_{\cC(\bb, B\ssm\{i\})} = \bmb \prod_{j\in B\setminus\{i\}}
		\left( \prod_{x\mid m_{\tau(j)}, \, x\nmid m_j} x \right).$$
		
		Thus the quotient is
		$$\frac{\bm_{\cC(\bb, B)}}{\bm_{\cC(\bb, B\ssm\{i\})}}=
		\prod_{x \mid m_{\tau(i)}, x \nmid m_i} x =  \frac{m_{\tau(i)}} {\gcd(m_i, m_{\tau(i)})}=
		\frac{\lcm(m_i, m_{\tau(i)})}{m_i}$$
		and $(1)$ follows. To see equality $(2)$, note that by
		\eqref{e:monomial-label} 
		$$\frac{\bm_{\cC(\bb-\bff_i+\bff_{\tau(i)}, B\ssm\{i\})}}{\bm_{\cC(\bb, B\ssm\{i\})}} =
		\frac{m_{\tau(i)}}{m_i} $$
		and thus using equality $(1)$,
		$$\frac{\bm_{\cC(\bb, B)}}{\bm_{\cC(\bb-\bff_i+\bff_{\tau(i)},
				B\ssm\{i\})}}= 
		\frac{m_{i}}{m_{\tau(i)}}
		\frac{\bm_{\cC(\bb, B)}}{\bm_{\cC(\bb, B\ssm\{i\})}} =
		\frac{\lcm(m_i, m_{\tau(i)})}{m_{\tau(i)}}.$$
	\end{proof}

As a result of \cref{p:simplify}, the homogenized differentials described in 
\eqref{res} can be written as

\begin{align}
\begin{split}
\label{e:res-again}
\partial_i(\bu_{\cC(\bb,B)})=
&\sum_{1 \leq k\leq i}(-1)^{k+1} \frac{\lcm(m_{j_k}, m_{\tau(j_k)})}{m_{j_k}}
\ \bu_{\cC(\bb,B\ssm \{j_k\})}\  \\
+& \sum_{1 \leq k\leq i}(-1)^{k} \frac{\lcm(m_{j_k}, m_{\tau(j_k)})}{m_{\tau(j_k)}}
\ \bu_{\cC(\bb-\bff_{j_k}+\bff_{\tau(j_k)},B\ssm\{j_k\})}\,.
\end{split}
\end{align}

\begin{example}\label{ex:path-3} 
  Let $G$ be the path graph in \cref{e:running1} and \cref{ex:path-0}.  In the cell complex
  $\ol{G^2}$ drawn in \cref{ex:path-2}, let
  $c=\ocC((0,1,1), \{1,2\})$ be the
  shaded square.  The cells of dimension $1$ are 
  the 6 line segments representing the edges, with the horizontal segments on the first line below and the vertical ones on the second.
\begin{align*}
c_1&=\ocC((1,1,0), \{1\}) & c_2&=\ocC((0,2,0), \{1\}) &
c_3&=\ocC((0,1,1), \{1\})\\
c_4&=\ocC((0,1,1), \{2\})& c_5&=\ocC((1,0,1),\{2\}) & c_6&=\ocC((0,0,2), \{2\}).
\end{align*} 
  The $0$-dimensional cells are shown below, where the first entry is $\varphi(2,0,0) = (0,0)$.
\begin{align*}
c_1'&=\ocC((2,0,0), \varnothing) & c_2'&=\ocC((1,1,0), \varnothing)&
c_3'&=\ocC((0,2,0),\varnothing)\\
c_4'&=\ocC((0,1,1), \varnothing)& c_5'&=\ocC((1,0,1),\varnothing) & c_6'&=\ocC((0,0,2), \varnothing).
\end{align*}  
 With this notation, the differential of the complex $C (\ol{G^2}, \sfk)$  is given as follows: 
\begin{align*}
\partial(\bu_c) \hskip.06in &=\bu_{c_2}+\bu_{c_4}-\bu_{c_3}-\bu_{c_5}&
\partial(\bu_{c_1})&=\bu_{c'_2}-\bu_{c'_1} \\
\partial(\bu_{c_2})&=\bu_{c'_3}-\bu_{c'_2}&
\partial(\bu_{c_3})&=\bu_{c'_4}-\bu_{c'_5}\\
\partial(\bu_{c_4})&=\bu_{c'_4}-\bu_{c'_3}&
\partial(\bu_{c_5})&=\bu_{c'_5}-\bu_{c'_2}\\ 
\partial(\bu_{c_6})&=\bu_{c'_6}-\bu_{c'_4} && &&
\end{align*}

The complex $C(\ol{G^2}, \sfk)$ is thus as follows: 
$$
0\to \sfk\xrightarrow{\begin{bmatrix}0\\1\\-1\\1\\-1\\0\end{bmatrix}}\sfk^6\xrightarrow{\begin{bmatrix}-1&0&0&0&0&0\\1&-1&0&0&-1&0\\0&1&0&-1&0&0\\0&0&1&1&0&-1\\0&0&-1&0&1&0\\0&0&0&0&0&1\end{bmatrix}}\sfk^6.
$$

  Now we homogenize this complex using the monomial generators of $I$: $m_0=xy$,
  $m_1= yz$, and $m_2= zu$, which yields:
$$
\mathbb F_{\ol{G^2}}: \quad 0\to R(-4)\xrightarrow{\begin{bmatrix}0\\u\\-y\\x\\-z\\0\end{bmatrix}}R(-3)^6\xrightarrow{\begin{bmatrix}-z&0&0&0&0&0\\x&-z&0&0&-u&0\\0&x&0&-u&0&0\\0&0&x&y&0&-u\\0&0&-z&0&y&0\\0&0&0&0&0&y\end{bmatrix}} R(-2)^6.
$$ and  we get a multigraded resolution as follows: 
$$
 \quad 0 \to  R(xy^2z^2u)\xrightarrow{\begin{bmatrix}0\\u\\-y\\x\\-z\\0\end{bmatrix}}\begin{array}{ccc} R(x^2y^2z)\\  \oplus R(xy^2z^2)\\ \oplus  R(xyz^2u) \\ \oplus R(y^2z^2u)\\ \oplus R(xy^2zu) \\ \oplus R(yz^2u^2) \end{array} \xrightarrow{\begin{bmatrix}-z&0&0&0&0&0\\x&-z&0&0&-u&0\\0&x&0&-u&0&0\\0&0&x&y&0&-u\\0&0&-z&0&y&0\\0&0&0&0&0&y\end{bmatrix}}  \begin{array}{ccc} R(x^2y^2) \\ \oplus R(xy^2z)\\ \oplus  R(y^2z^2) \\ \oplus R(yz^2u)\\ \oplus R(xyzu) \\ \oplus R(z^2u^2) \end{array}.
$$
\end{example}

\begin{remark}
Let $r\ge 1$. In Section~\cref{ss:covering}, we introduced the translations $t_i\colon \ol{G^r}\to \ol{G^{r+1}}$ and we noted they are cellular maps. At the level of the associated homogenized chain complexes, these cellular maps induce chain maps $\widetilde{t_i}\colon \mathbb F_{\ol{G^r}}\to \mathbb F_{\ol{G^{r+1}}}$ described by 
$$
\widetilde{t_i}(\bu_{\cC(\bb,B)})=\bu_{\cC(\bb+\bff_i,B)},
$$
for all $\ocC(\bb, B)\in \ol{G^r}$ and all $i$ with $0\le i\le q$.  One can check that these maps are indeed chain maps, using the description of the differentials in \cref{e:res-again}.   As will be shown in \cref{s:minres}, $F_{\ol{G^r}}$ is a minimal free resolution of $I^r$ and $F_{\ol{G^{r+1}}}$ is a minimal free resolution of $I^{r+1}$. 

Note that the  chain maps $\widetilde t_i$ can also be described as the chain maps induced by the map $I^r\to I^{r+1}$ given by multiplication by $m_i$.  Additionally, if we consider the map $\psi\colon \bigoplus_{i=0}^q I^r\to I^{r+1}$ whose $i^{th}$ component is given by multiplication by $m_i$ for each $i$, then the induced chain maps
 $$\bigoplus_{i=0}^q \widetilde t_i\colon \bigoplus_{i=0}^q \mathbb F_{\ol{G^r}}\to 
\mathbb F_{\ol{G^{r+1}}}
$$
 are surjective when $r\ge q$, in view of  \cref{G^r-cover}.  Hence, the induced  maps 
 $$\Tor_j^R(\psi, \sfk)\colon \Tor_j^R(\bigoplus_{i=0}^q I^r, \sfk)\to \Tor_j^R( I^{r+1},\sfk)$$ are surjective for all $j\ge 0$ and all $r\ge q$. 
\end{remark}

\section{Minimal cellular resolutions of $I^r$}
\label{s:minres}

In this section, we construct a cellular
resolution of $I^r$,  where $I$ is a monomial ideal of projective
  dimension one with a minimial resolution supported on a graph
  $G$, showing that \cref{q:driving} has a positive answer for this class of ideals. At this point we have constructed $\ol{G^r}$ as  a
  polyderal cell complex labelled with the monomials that generate
  $I^r$. What remains in this section is to prove that the labelled
  chain complex of $\ol{G^r}$ gives a minimal free resolution of
  $I^r$. To complete this task we will compare the labelled chain
  complex of $\ol{G^r}$ to another chain complex, which is in turn
isomorphic to the minimal free resolution of $I^r$.

The {\it Rees algebra} of $I$ is a well-studied object that encodes
all powers of $I$ into a single  construction. Having a family of
minimal resolutions of $I^r$ for all $r>0$ that are constructed from a
common base, namely $G$, naturally leads one to wonder if these
resolutions are related to a resolution of the Rees algebra. Indeed,
the  other chain complex mentioned in the previous paragraph
turns out to stem from the Rees algebra.   We will show
that for square-free monomial ideals of projective dimension one, the
Rees algebra can be presented as a quotient of a polynomial ring by a
complete intersection ideal. A strand of the Koszul resolution of the
complete intersection is the chain complex that will be
isomorphic  to the chain complex of $\ol{G^r}$.

 \begin{bfchunk}{Rees algebras and  ideals of linear type.}
  We first recall background information related to the Rees algebra and show that square-free monomial ideals of projective dimension one are of  linear type. 
  
  Let $(R, \fm)$ be a Noetherian local ring and $I$ an ideal of
  $R$. The { \bf Rees algebra} of $I$, denoted $R[It]$, is a subalgebra
  of the polynomial ring $R[t]$ consisting of polynomials for which
  the coefficient of $t^s$ is in $I^s$ for all $s$. That is,
 $$R[It]=R \oplus It \oplus I^2t^2 \oplus I^3t^3 \oplus \cdots.$$ When
  $R=\sfk[x_1,\ldots,x_n]$ is a polynomial ring, the definition
  carries through with $\fm = (x_1, \ldots, x_n)$ the homogeneous
  maximal ideal of $R$. A common way to gain insight into $R[It]$ is
  to embed it in a larger polynomial ring and then study its defining
  ideal. Suppose $I$ is our monomial ideal of projective dimension one
  minimally generated by square-free monomials $m_0, \ldots, m_q$, and let
  $S=R[T_0, \ldots, T_q]$. The map $\psi : S \rightarrow R[It]$
  defined by $\psi(T_i) = m_it$ gives a surjective $R$-algebra
  homomorphism. Hence, if $J={\mbox{\rm ker}}\psi$, then
  $$R[It]= S/J = R[T_0, \ldots, T_q]/J.$$

  The map  $\psi$ is graded so its kernel $J=J_1 +J_2 + \cdots$
  is  a homogeneous ideal, commonly referred to as the  {\bf
    defining ideal} of $R[It]$.  The graded component of $J$ of
    degree $1$, namely $J_1$, is generated by elements of the
  form: $$\{a_0T_0+a_1T_1+\cdots + a_qT_q \mid a_0m_0 + \ldots +a_qm_q
  =0 \} $$ which correspond to the generators of the first syzygy module 
   of $I$. In this specific case, where
    $I=(m_0,\ldots,m_q)$ has projective dimension one, we know that
    (\cite[proof of Theorem~8]{FH}) the first syzygy module can be
    generated by the homogenized generators of the chain complex of
    the graph supporting the resolution of $I$, namely
    \begin{equation}
\label{e:regseq}
g_k=\frac{\lcm(m_k, m_{\tau(k)})}{m_k}T_k-\frac{\lcm(m_k,
  m_{\tau(k)})}{m_{\tau(k)}}T_{\tau(k)} \qfor k\in [q].
\end{equation}
 Therefore, $J_1=(g_1,\dots,g_q)$.
 \end{bfchunk}
    
  We now show that $I$ is of {\bf linear type}, that is, $J$ is
  generated by its degree one elements.

  \begin{lemma}\label{l:num-gens}
 Let $I$ be an ideal of projective dimension one
 minimally generated by $q+1$ square-free monomials in the polynomial
 ring $R=\sfk[x_1,\ldots,x_n]$ over a field $\sfk$.  Then $q+1 \leq
 n$.
\end{lemma}

  \begin{proof}
    If $q=0$ the result follows trivially, so suppose $q\geq 1$.  From
    the concrete construction of the tree $G$ that supports a minimal
    free resolution of $I$ it follows~\cite[(4.0.1)]{CEFMMSS} that for
    every $i \in [q]$ there exists a variable $x_{a_i} \in \{x_1,
    \ldots, x_n\}$ such that
    \begin{equation}\label{e:vars}
      x_{a_i} \nmid
    m_i \qand x_{a_i}\mid m_j \qfor 0\leq j < i.
    \end{equation}

    This implies that $q \leq n$. On the other hand,
    by \eqref{e:vars} $x_{a_1}\cdots x_{a_q} \mid m_0$. If for any $j$ $x \mid m_j$ implies $x = x_{a_i}$ for some $i$, then since $I$ is square-free, $m_j \mid m_0$, a contradiction. Thus we must have at least $q+1$ variables; i.e. $q+1
    \leq n$.
\end{proof}

\begin{theorem}\label{t:lineartype}
Let $I$ be a square-free monomial ideal of projective dimension one in
the polynomial ring $R=\sfk[x_1,\ldots,x_n]$ over a field $\sfk$.
Then $I$ is of linear type. 
\end{theorem}

\begin{proof} 
   Let $\mu(I)$ denote the minimal number of generators of $I$. If
    we show that $\mu(I_p) \leq \dpth R_p$ for all prime ideals $p$ of
    $R$ containing $I$, then by a result of Tchernev~\cite[Theorem
      5.1]{Tc} $I$ is of linear type.
      
    Let $p$ be a prime ideal of $R$ containing $I$, and suppose $p'$
    is the ideal generated by all the variables in $p$, i.e. $p'=(x_i
    \mid x_i \in p) \subseteq p$, and suppose that $p'$ is generated
    by $n'$  variables. It is not difficult to see
    that (e.g. see~\cite[proof of Lemma 1]{F02}) $I_p$ and $I_{p'}$
    have the same (square-free) monomial generating set, and so
    \begin{equation}\label{e:mu}
      \mu(I_p) = \mu(I_{p'}).
    \end{equation}
   These generators form a square-free monomial ideal in the polynomial
   ring $R'$ generated by the $n'$ variables generating $p'$, and
   $\pd_{R'}(I_{p'}) \leq 1$, and so from \cref{l:num-gens} it follows
   that
    \begin{equation}\label{e:mu'}
      \mu(I_{p'}) \leq n'= \htt p' \leq \htt p = \dpth R_p.
    \end{equation}
    \cref{e:mu,e:mu'} together imply $\mu(I_p) \leq \dpth R_p$ for all
    primes $I \subseteq p$, and we are done. 
  \end{proof}

\cref{t:lineartype} tells us that $$R[It] = S/J_1=S/(g_1,\dots,g_q).$$
Next we show that the generators  $g_1,\ldots,g_k$ listed in
\eqref{e:regseq} form a regular sequence, and make use of the
associated Koszul complex $\mathbb K$ to extract a chain complex that
leads to a minimal free resolution of $I^r$, where $r>0$.

\begin{theorem}
\label{t:regular}
Let $I$ be a square-free monomial ideal of projective dimension one in a polynomial ring $R$ over a field. Then
$$R[It]\cong R[T_0, \dots, T_q]/(g_1, \dots, g_{q})$$ where $g_1,
\dots, g_q$ are as in \eqref{e:regseq} and form a regular sequence.
\end{theorem} 

\begin{proof}

Recall that the
Rees algebra $R[It]$ has dimension equal to $\dim R+1=n+1$. In
particular, $$ \text{height}(J_1)=\dim R[T_0, \dots, T_q]-\dim
R[It]=n+q+1-(n+1)=q\,.
$$
This yields the desired conclusion.  
\end{proof}

\cref{t:regular} shows that the generators of $J_1$ form a regular sequence in
$S=R[T_0, \dots, T_q]$ and as a result, the Koszul complex of $(g_1,
\ldots, g_q)$ over the ring $S$ is a minimal free resolution of
$S/(g_1, \ldots, g_q)$ (see~\cite{BH} for basic facts about exterior
algebras and Koszul complexes).  Since $S/(g_1, \ldots, g_q)= R[It]$,
we gain information about the cellular resolutions of powers of
$I$ by getting information on the graded strands of the Koszul complex
of $S/(g_1, \ldots, g_q)$. To make that information easier to access,
we use the fact that exterior algebras commute with base
extensions (see, e.g., \cite[\S 6.13]{NJ}).  In other words, instead of starting with
a rank $q$ free module over $S$, we let $F$ be a rank $q$ free module
over $R$ with basis $e_1, \ldots e_q$.

The augmented  Koszul complex on the elements $g_1, \dots,g_q$ is 
\begin{align*}
0\to \bw^{q} F(-q)\otimes S \to &\bw^{q-1} F(-q+1)\otimes S \to \cdots \\
\dots \to \bw^1&F(-1)\otimes S \to S \to S/{J_1} \to 0.
\end{align*}
In this complex, the degree shifts refer to the $T$-grading on $J$, which is determined by an element's total degree in $T_0, \ldots ,T_q$. We write $S=S_0 \oplus S_1 \oplus S_2 \oplus \cdots$ where $S_i$ is the $i$th graded component of $S$ relative to $T$. We then take the linear strand of the Koszul complex above with $T$-degree equal to $r$, we obtain the complex 
\begin{align*} 
\mathbb K^r: \qquad 0\to \bw^{r} F(-r) \otimes S_0 \to \bw^{r-1} F(-r+1)\otimes S_1 \to  \ldots \\
  \dots \to \bw^1 F(-1)\otimes S_{r-1} \to S_r \to I^r \to 0.
\end{align*} 
We note that $\bw^{i} F= 0 $ for all $i>q$.
The differential of this complex can be described by 
$$
\partial^{\mathbb K^r}_i(e_{j_1}\wedge \dots \wedge e_{j_i}\otimes w)=\sum_{k=1}^i (-1)^{k-1}e_{j_1}\wedge \dots\wedge \widehat{e_{j_k}}\wedge\dots\wedge e_{j_i}\otimes g_{j_k}w.
$$ where  $w \in S_{r-i}$.

\begin{bfchunk}{Isomorphism of complexes.} 
We now  compare $\mathbb K^r$ to the $\mathbb Z^n$-graded
cellular complex $\mathbb F_{\ol{G^r}}$ established in
 \eqref{e:cell-complex} with differentials given in \cref{res}.  We
prove that both complexes are isomorphic, which gives us that the
labeled chain complex $\ol{G^r}$ supports the minimal free resolution
of $I^r$.
\end{bfchunk}

\begin{proposition}
\label{p:isomcomplexes}
The chain complexes $\mathbb F_{\ol{G^r}}$ and $\mathbb K^r$ are isomorphic. 
\end{proposition} 

\begin{proof}
To describe an isomorphism $\rho\colon \mathbb F_{\ol{G^r}}\to \mathbb K^r$, we need to specify how it acts on the basis elements of the free modules in $\mathbb F_{\ol{G^r}}$. Let $\bu_{\cC(\bb,B)}$  be a basis element of $F_i$, as in \cref{e:cell-complex}, with $B =\{j_1, \dots, j_i\}$ and $
j_1<j_2<\dots<j_i$. Since $\bb \in \cN_r$, $b_0+\cdots +b_q=r$. Let  ${\bf{T}^{\bb}}=T_0^{b_0}\cdots T_q^{b_q} \in S_{r}$. Define $\bb' = \bb - \sum_{i \in B} \bff_i$. Then ${\bf{T}^{\bb'}} \in S_{r-i}$.
We define 
$$
\rho(\bu_{\cC(\bb,B)})=e_{j_1}\wedge \dots \wedge e_{j_i}\otimes {\bf T}^{\bb'}\in \bw^i F\otimes S_{r-i}.
$$
We also note that 
$$\rho(\bu_{\cC(\bb, B\ssm \{j_k\})})=e_{j_1}\wedge \dots \wedge \widehat{e_{j_k}}\wedge\dots\ \wedge e_{j_i}\otimes {\bf T}^{\bb'}T_{j_k}\in \bw^{i-1} F\otimes S_{r-i+1}$$ and hence $$
\rho(\bu_{\cC(\bb,B\ssm \{j_k\})})=e_{j_1}\wedge \dots \wedge \widehat{e_{j_k}}\wedge\dots\ \wedge e_{j_i}\otimes {\bf T}^{\bb'+\bff_{j_k}}.$$
The map $\rho$ is clearly bijective. We need to verify now that $\rho$ is indeed a homomorphism of complexes (i.e. commutes with the differential). Setting $B=\{j_1, \dots, j_i\}$  in \cref{e:res-again}, we get: 
  \begin{align*}
\partial_i^{\mathbb F_{\ol{G^r}}}(\bu_{\cC(\bb,B)})=\sum_{1 \leq k\leq i}(-1)^{k+1} &\frac{\lcm(m_{j_k}, m_{\tau(j_k)})}{m_{j_k}} \bu_{\cC(\bb,B\ssm \{j_k\})}+\\
&+\sum_{1 \leq k\leq i}(-1)^{k}\frac{\lcm(m_{j_k}, m_{\tau(j_k)})}{m_{\tau(j_k)}} \bu_{\cC(\bb-\bff_{j_k}+\bff_{\tau(j_k)},B\ssm\{j_k\})}\,.
\end{align*}
We then have 
\begin{align*}
\rho \left(\partial_i^{\mathbb F_{\ol{G^r}}}(\bu_{\cC(\bb,B)})\right)&=
\sum_{1 \leq k\leq i}(-1)^{k+1}e_{j_1}\wedge \dots\wedge \widehat{e_{j_k}}\wedge\dots\wedge e_{j_i}\otimes \frac{\lcm(m_{j_k}, m_{\tau(j_k)})}{m_{j_k}} {\bf T}^{\bb'+{\bf f}_{j_k}}+\\
&\sum_{1 \leq k\leq i}(-1)^{k}e_{j_1}\wedge \dots\wedge \widehat{e_{j_k}}\wedge\dots\wedge e_{j_i}\otimes \frac{\lcm(m_{j_k}, m_{\tau(j_k)})}{m_{\tau(j_k)}}{\bf T}^{\bb'-\bff_{j_k}+\bff_{\tau(j_k)}+ \bff_{j_k}}.
\end{align*}
On the other hand, we have: 
\begin{align*}
&\partial_i^{\mathbb K^r}(\rho(\bu_{\cC(\bb,B)})=\partial_i^{\mathbb K^r}(e_{j_1}\wedge \dots \wedge e_{j_i}\otimes {\bf T}^{\bb'})\\&=\sum_{1 \leq k\leq i} (-1)^{k-1}e_{j_1}\wedge \dots\wedge \widehat{e_{j_k}}\wedge\dots\wedge e_{j_i}\otimes \left(\frac{\lcm(m_{j_k}, m_{\tau(j_k)})}{m_{j_k}}T_{j_k}-\frac{\lcm(m_{j_k}, m_{\tau(j_k)})}{m_{\tau(j_k)}}T_{\tau(j_k)}\right){\bf T}^{\bb'}\\
&=\rho\left(\partial_i^{\mathbb F_{\ol{G^r}}}(\bu_{\cC(\bb,B)})\right)\,.
\end{align*}
\end{proof}

\begin{bfchunk} {Main results.}
So far, we have seen in \cref{t:regular} that $g_1, \dots, g_q$ is a regular
sequence and the Koszul complex on these elements are
acyclic. Furthermore, since the number of generators of $J_1$ is less than the dimension of $R$, the strand $\mathbb K^r$ of
$\mathbb K$ a minimal free resolution of the $r^{th}$ graded piece
  of $R[It]\cong S/J_1$, which is $I^r$. From this, we obtain the result below.
\end{bfchunk}

\begin{theorem}
\label{thm:main}
Let $I$ be a square-free monomial ideal of projective dimension one,
and $r>0$. Then $I^r$ has a cellular minimal free resolution
  supported on the polyhedral cell complex $\ol{G^r}$.
\end{theorem}

\begin{proof}
This follows immediately from \cref{p:isomcomplexes}.
\end{proof}

\begin{corollary}[{\bf The projective dimension of $I^r$ and $I^r/I^{r+1}$}]\label{c:pd-I^r}   If $I$ is generated by $q+1$ square-free
  monomials in the polynomial ring $R$, $I$ has projective dimension
  one, and $r$ is a positive integer, then
    \begin{equation*}
  \pd_RI^r=\begin{cases} q & \qif r\ge q\\
                          r & \qif r< q
\end{cases} \qand \pd_RI^r/I^{r+1}=\begin{cases} q+1 & \qif r\ge q-1\\
                          r+2 & \qif  r< q-1\end{cases}\,.
\end{equation*}
\end{corollary}
 \begin{proof} 
 The projective dimension of $I^r$ can be read from the complex $\mathbb K^r$ (which is isomorphic to $\mathbb F_{\ol{G^r}}$), as we note that $\bw^i F=0$ if and only if $i>q$. 
 
To obtain the formula for $\pd_R I^r/I^{r+1}$ one uses the fact that the map 
$$
\Tor_i^R(I^{r+1}, \sfk)\to \Tor_i^R(I^r, \sfk)
$$
induced by the inclusion $I^{r+1}\subseteq I^r$ is zero for all $i\ge 1$; this follows from a result of Maleki \cite[Proposition 3.5]{Ma19}.  The map above can be seen to be zero when $i=0$ as well. For any $i\ge 1$ there is thus a short exact sequence 
$$
0\to \Tor_i^R(I^r, \sfk)\to \Tor_i^R(I^{r}/I^{r+1}, \sfk)\to \Tor_{i-1}^R(I^{r+1}, \sfk)\to 0
$$
that gives $\pd_R I^r/I^{r+1}=\max\{\pd_R I^r, 1+\pd_R I^{r+1}\}=1+\pd_R I^{r+1}$. 
\end{proof}

Given the structure of this resolution, we are able to say more. We
can find the precise Betti numbers for each power of $I$.

 \begin{corollary}[{\bf The Betti numbers of $I^r$}]\label{c:betti-I^r}
   If $I$ is generated by $q+1$ square-free monomials in the
    polynomial ring $R$ and $I$ has projective dimension one, then the 
 $t^{th}$ Betti number of $I^r$ is ${{q} \choose {t}} \cdot {{q+r-t} \choose {r-t}}$ if $t \leq r$ and $0$ otherwise. In particular, the Betti numbers of $I^r$ do not depend on the characteristic of the base field.
\end{corollary}
\begin{proof}
If $\bb \in G^r$ then each distinct $B \subseteq \supp(\bb)$ with $|B|=t$ determines a cell of size $t$ embedded as above. For each $t$, there are ${{q} \choose {t}}$ distinct sets $B$ of size $t$. For each such $B$, there are ${{q+r-t} \choose {r-t}}$ vertices of $G^r$ whose support contains $B$.
\end{proof}

\begin{example} In our running \cref{ex:path-0}, the path has three vertices, and $I=(xy,yz,zu)$. By Corollary \ref{c:pd-I^r}, $\pd_SI^r=q=2$ for all $r\geq 2$. Furthermore, applying Corollary \ref{c:betti-I^r} for $r=3$, we obtain $\beta_0(I^3)= {5 \choose 3}=10$, $\beta_1(I^3)=2\cdot {4 \choose 2}=12$, and $\beta_2(I^3)= {3 \choose 1}=3$. 
\end{example}

\begin{bfchunk}{An application to the fiber cone.} 
While
the result below also follows directly from~\cref{t:regular}, we
present an alternate proof that illustrates additional properties of the ideals $I$. Before stating the
result, we briefly recall the relevant definitions and
background. Additional information can be found in
\cite{Wolmerbook, Vi}.
 
 An ideal $J\subseteq I$ is a {\bf reduction}~\cite{NR} of $I$ if
 there exists an integer $r$ such that $JI^r = I^{r+1}$. Reductions
 can be viewed as approximations of an ideal $I$ that share asymptotic
 behavior but have fewer generators. The {\bf analytic spread} of $I$,
 denoted by $\ell(I)$, is the minimal number of generators of a
 minimal reduction of $I$.

 An interesting class of ideals are those that are their own minimal
   reductions. Since the dimension of the {\bf fiber cone of $I$}
   $${\mathcal F}(I)= {\frac{R[It]}{\fm R[It]}} = R/\fm \oplus I/\fm I
   \oplus I^2/\fm I^2 \oplus \cdots $$ is the analytic spread, the
   fiber cone is often used to detect this property.
 \end{bfchunk}

\begin{corollary}
If $I$ is a square-free monomial ideal of projective dimension one
minimally generated by $(m_0, \ldots, m_q)$, then $I$ has no proper
non-trivial reductions. Moreover, the fiber cone is isomorphic to a
polynomial ring in $q$ variables.
\end{corollary}

\begin{proof}
Consider the standard presentation map
$$\phi: R[T_0, \ldots,T_q] \rightarrow R[It] \mbox{\rm{ induced by }}
\phi(T_i) = m_i t.$$ Since $I$ is a monomial ideal, $J=\ker(\phi)$ is
a binomial ideal. In general, $J(0)=(J+\m)/\m R[It]$ is the defining
ideal of $\mathcal{F}(I)$. That is, $$\mathcal{F}(I) \cong \sfk[T_0,
  \ldots, T_q]/J(0).$$ Thus the generators of $J(0)$ have the form
$T_0^{a_0}\cdots T_q^{a_q} - T_0^{b_0}\cdots T_q^{b_q}$ where
$\phi(\bf{T^a}) = \phi(\bf{T^b})$. By definition of $\phi$, this
implies $|\ba| = |\bb|$ and $\bma = \bmb$, which by  \cite[Proposition 4.1]{CEFMMSS} means
that $\ba = \bb$ and therefore $J(0)=(0)$. Thus $$\mathcal{F}(I) \cong
\sfk[T_0, \ldots , T_q].$$ It is well known
(see, for example, \cite{Wolmerbook}) that $\ell(I) = \dim
\mathcal{F}(I)$ and $\ell(I) \leq \mu(I)$, where $\mu(I)$ is the
minimal number of generators of $I$. Hence $\ell(I) = \dim
\mathcal{F}(I)=q=\mu(I)$, which implies $I$ has no proper non-trivial
reductions.
\end{proof}

\subsubsection*{Acknowledgements} 

The research leading to this paper was initiated during the weeklong
workshop ``Women in Commutative Algebra'' (19w5104) which took place
at the Banff International Research Station (BIRS). The authors would
like to thank the organizers and acknowledge the
hospitality of BIRS and the additional support provided by the
National Science Foundation, DMS-1934391.

For this work Liana \c Sega was supported in
part by a grant from the Simons Foundation (\#354594), and Susan Cooper and Sara Faridi were
supported by Natural Sciences and Engineering Research Council of
Canada (NSERC).

The authors are grateful to Bernd Ulrich, Claudia Pollini, Alexandra Seceleanu for useful comments
regarding this work. The computations for this project were done using
the computer algebra program Macaulay2~\cite{M2}.

For the last author, this material is based upon work
  supported by and while serving at the National Science
  Foundation. Any opinion, findings, and conclusions or
  recommendations expressed in this material are those of the authors
  and do not necessarily reflect the views of the National Science
  Foundation.


\bibliography{bibliography}

\end{document}